\documentclass[12pt]{amsart}

\usepackage[T1]{fontenc}
\usepackage{lmodern}

\usepackage{amscd, amssymb, amsmath, amsthm}
\usepackage{enumerate, latexsym, mathrsfs}
\usepackage{xypic}

\usepackage[pdftex,lmargin=1in,rmargin=1in,tmargin=1in,bmargin=1in]{geometry}
\usepackage[
		bookmarks=true, bookmarksopen=true,%
    bookmarksdepth=3,bookmarksopenlevel=2,%
    colorlinks=true,%
    linkcolor=blue,%
    citecolor=blue]{hyperref}
\newtheorem{theorem}{Theorem}[section]
\newtheorem{lemma}[theorem]{Lemma}
\newtheorem{proposition}[theorem]{Proposition}
\newtheorem{corollary}[theorem]{Corollary}

\theoremstyle{definition}

\newtheorem{declaration}[theorem]{Declaration}

\newtheorem{remark}[theorem]{Remark}

 \numberwithin{equation}{section}

\newcommand{\Natural}{{\mathbb N}}
\newcommand{\Real}{{\mathbb R}}

\newcommand{\Complex}{{\mathbb C}}
\newcommand{\Integral}{{\mathbb Z}}

\newcommand{\Hyp}{{\mathbf{H}^3}}

\newcommand{\SO}{{\mathrm{SO}}}
\newcommand{\PSL}{{\mathrm{PSL}}}

\newcommand{\ocurves}{{\mathbf{\Gamma}}}

\newcommand{\opants}{{\mathbf{\Pi}}}

\newcommand{\ocobordism}{{\mathbf{\Omega}}}

\newcommand{\chlen}{{\mathbf{hl}}}

\title[QF subsurface of odd Euler characteristic]{Immersing quasi-Fuchsian surfaces of odd Euler characteristic in closed hyperbolic 3-manifolds}

\author[Yi Liu]{%
        Yi Liu} 
\address{%
        Beijing International Center for Mathematical Research\\
				No.~5 Yiheyuan Road, Haidian District, Peking University\\
				Beijing 100871, China P.R.} 
\email{%
    liuyi@math.pku.edu.cn}

\thanks{Supported by the Recruitment Program of Global Youth Experts of China}
\subjclass[2010]{Primary 57M50; Secondary 57M10, 30F40}
\keywords{good pants, quasi-Fuchsian, homological torsion growth}

\date{%
 \today}

\begin{document}

\begin{abstract}
	In this paper, it is shown that every closed hyperbolic 3-manifold contains
	an immersed quasi-Fuchsian closed subsurface of odd Euler characteristic.
	The construction adopts the good pants method, and the primary new ingredient
	is an enhanced version of the connection principle, which allows one to connect
	any two frames with a path of frames 
	in a prescribed relative homology class of the frame bundle.
	The existence result is applied to show that every uniform lattice of $\PSL(2,\Complex)$
	admits an exhausting nested sequence of sublattices with exponential homological torsion growth.
	However, the constructed sublattices are not normal in general.
\end{abstract}

\maketitle

\section{Introduction}
	For any arbitrary closed hyperbolic $3$-manifold,
	immersed quasi-Fuchsian closed subsurfaces 
	have been constructed by J.~Kahn and V.~Markovic \cite{KM-surfaceSubgroup}.
	While their subsurfaces are always orientable,
	it is possible to build non-orientable subsurfaces of even Euler characteristic
	by modifying their construction, (see \cite{Sun-virtualTorsion} for example).
	It remains to be an open question 
	whether a subsurface can be constructed to have odd Euler characteristic,
	\cite[Section 11, Question 7]{Agol-proceedingsICM}.
	The question finds its motivation in several aspects of
	hyperbolic $3$-manifold geometry,
	including all the results we state in the introduction.
	The theme result of this paper
	is an affirmative answer to the existence question:
	
	\begin{theorem}\label{main-oddEulerCharacteristic}
		Let $M$ be a closed hyperbolic $3$-manifold. Then there exists 
		a connected closed surface of odd Euler characteristic $\Sigma$
		which admits
		a $\pi_1$--injective, quasi-Fuchsian immersion $j\colon \Sigma\looparrowright M$.
	\end{theorem}
		
	When an orientable closed $3$-manifold contains an embedded closed
	subsurface of odd Euler characteristic, the manifold admits a degree-one map onto
	the real projective $3$-space, by a characterization of 
	C.~Hayat-Legrand, S.~Wang, and H.~Zieschang 
	\cite[Theorem 4.1]{HWZ-lens}.	
	This fact provides one reason for us to be interested 
	in the virtual existence of such subsurfaces.
	Since hyperbolic $3$-manifold groups are LERF \cite[Theorem 9.2]{Agol-VHC},
	Theorem \ref{main-oddEulerCharacteristic} has the following immediate consequence:
	
	\begin{corollary}\label{dominatingProjectiveSpace}
		Every closed hyperbolic $3$-manifold has an orientable finite cover which admits a degree-one map onto
		the real projective $3$-space.
	\end{corollary}

	Corollary \ref{dominatingProjectiveSpace} might be
	a toy case of some more general fact.
	In \cite{Sun-virtualDomination}, H.~Sun has proved that closed hyperbolic $3$-manifolds virtually $2$-dominates
	any other orientable closed $3$-manifolds. In other words, for any closed orientable $3$-manifold
	$N$, every closed hyperbolic $3$-manifold $M$	has an orientable finite cover which admits a map onto $N$ of degree $2$.
	It seems that with the improved techniques of this paper (Theorem \ref{gradedConnectionPrinciple}),
	there is a good chance to promote Sun's result to virtual $1$-dominations. 
	
	Another curious application of Theorem \ref{main-oddEulerCharacteristic} shows
	certain exponential torsion growth for uniform lattices of $\PSL(2,\Complex)$:
	
	\begin{theorem}\label{exponentialTorsionGrowth}
		Every uniform lattice $\Gamma$ of $\PSL(2,\Complex)$
		contains an exhausting nested sequence of 
		non-normal torsion-free sublattices $\{\Gamma_i\}_{i\in\Natural}$
		such that
		$$\liminf_{n\to\infty} \frac{\log |H_1(\Gamma_n;\,\Integral)_{\mathrm{tors}}|}{[\Gamma:\Gamma_n]}\,>\,0.$$		
	\end{theorem}
	
	Here the sequence being nested means that each subgroup $\Gamma_n$ 
	contains its successor $\Gamma_{n+1}$,
	and being exhausting means that the common intersection of all the subgroups is trivial.
	The term in the logarithmic function is the cardinality of the torsion subgroup of
	the first integral homology of the group $\Gamma_n$.
	
	For general uniform lattices of $\PSL(2,\Complex)$, Theorem \ref{exponentialTorsionGrowth}
	seems to be the first known 
	exhausting sequence of sublattices
	with exponential homological torsion growth.
	A key feature of the construction is that it produces
	a huge number of disjointly embedded subsurfaces of odd Euler characteristic
	in the quotient hyperbolic $3$-manifolds $\Hyp\,/\,\Gamma_n$,
	so a positive portion of $\Integral_2$ homological torsion can
	be recognized.
	As to be explained in more details in Section \ref{Sec-irregularExhaustingTower},
	our approach 
	is conceptually simple but technically nontrivial.
	Besides Theorem \ref{main-oddEulerCharacteristic}, 
	the construction relies essentially on the virtual specialness of hyperbolic $3$-manifold groups,
	proved by I.~Agol \cite{Agol-VHC} and D.~Wise \cite{Wise-book}.
	In particular, the sublattices are constructed inductively
	by invoking Wise's Malnormal Special Quotient Theorem.
	The statement of Theorem \ref{exponentialTorsionGrowth} is known to be false
	for certain uniform lattices of higher-rank simple real Lie groups, 
	including	$\mathrm{SL}(n,\Real)$ for $n>2$ and $\mathrm{SO}(p,q)$ for $q>1$ and large $p$,
	\cite{Abert--Gelander--Nikolov}.
	
	It should be pointed out, however, that Theorem \ref{exponentialTorsionGrowth}
	does not say much about the well known asymptotic growth conjecture on virtual homological torsion,
	(see \cite[Conjecture 1.13]{Bergeron-Venkatesh} for arithmetic lattices
	and \cite[Question 13.73]{Lueck-book} for closed Riemannian manifolds).
	One (optimistic) version of the conjecture may be stated as follows: 
	Given a lattice $\Gamma$ of $\PSL(2,\Complex)$,
	for all exhausting nested sequences of normal torsion-free sublattices of $\Gamma$,
	the limit in Theorem \ref{exponentialTorsionGrowth} equals $\mathrm{Vol}(\Hyp\,/\,\Gamma)\,/\,6\pi$.
	Even if $\Gamma$ is uniform,
	the sublattices $\{\Gamma_n\}_{n\in\Natural}$ of our construction
	are far from normal. In fact, as $n$ increase,
	larger and larger balls emerge in the quotient hyperbolic $3$-manifolds, 
	but in a highly scattered fashion with relatively small total volume,
	so the sequence	is
	not even convergent in the Benjamini--Schramm sense.
	
	The major innovation of this paper
	is an enhanced version of 
	the connection principle in good pants constructions,
	Theorem \ref{gradedConnectionPrinciple}.
	The reader is referred to Section \ref{Sec-preliminaries} for 
	terminology and background on this topic.
	To illustrate the primary issue,
	suppose that we are asked to construct a good curve $\gamma$ 
	so that it bounds a surface built up with good pants.
	It is known that we can design the construction
	to make $\gamma$ good and null homologous in $M$, 
	the oriented closed hyperbolic $3$-manifold under consideration.
	Even so, there is still one geometric obstruction.
	As a good curve, $\gamma$ has a so-called canonical lift $\hat{\gamma}$ in 
	the frame bundle $\SO(M)$,
	which is a loop of frames well defined up to homotopy. 
	For the null homologous good curve
	$\gamma$ to bound a good panted subsurface, it turns out that
	its canonical lift $\hat{\gamma}$ must also be null homologous in $\SO(M)$.
	Therefore, we need to seek for some refinement of our construction
	to make sure that the homology class of the canonical lift can be controlled.
	A solution of the example task above contains the key idea
	of our refined construction.
		An outline is deferred to the last part of the introduction.
	
	It turns out that the construction for Theorem \ref{main-oddEulerCharacteristic}
	can be reduced to a very similar situation:
	We need to find a `semi-good' curve $\sqrt{\gamma}$ whose double cover 
	is a bounding good curve $\gamma$.
	In fact, the presence of such a curve $\gamma$ is necessary since a good panted subsurface
	of odd Euler characteristic contains an odd number of good pants, so the gluing map as a free involution
	on the union of cuffs must preserve some component by the parity.
	One could produce the semi-good curve $\sqrt{\gamma}$ by an \textit{ad hoc} modification
	of the construction of the example task. Alternatively, as we present in this paper,
	we can establish a more generally adaptable construction,
	which allows us to connect frames by good paths of frames
	in any prescribed relative homology class (Theorem \ref{gradedConnectionPrinciple}).
	Accompanied with the enhanced version of the connection principle,
	the package of \cite{Liu-Markovic} can be applied
	to produce good curves and panted $2$-complexes in closed hyperbolic $3$-manifolds,
	now with significantly more precise control of their homology.
	For simplicity, we state the following ordinary form of Theorem \ref{gradedConnectionPrinciple},
	which roughly says that every integral first homology class of 
	the frame bundle can be represented 
	by a good curve, eventually, as we increase the required length.

	\begin{theorem}\label{constructGradedGoodCurves}
		Let $M$ be an oriented closed hyperbolic $3$-manifold.
		Given any homology class $\Xi\in H_1(\SO(M);\Integral)$ and any constant $\epsilon>0$, 
		there exists a constant $R_0=R_0(\epsilon,M,\Xi)>1$ such that the following holds true.
		For every constant $R>R_0$,
		there exists an $(R,\epsilon/R)$--curve 
		$\gamma$ in $M$ of which the canonical lift $\hat\gamma$
		represents $\Xi$ in $H_1(\SO(M);\Integral)$.
	\end{theorem}		
		
	We remark that fixing $\epsilon$ and $M$,  
	the \textit{a priori} bound $R_0(\epsilon,M,\Xi)$ can be chosen to depend linearly on the word length of $\Xi$, 
	with respect to any finite generating set of the first homology.
	As $R$ tends to infinity, the number of $(R,\epsilon/R)$--curves 
	representing any given class $\Xi$ grows exponentially fast,
	in a fashion independent of $\Xi$.
	In fact, it is possible to derive the asymptotics 
	from Kimoto--Wakayama \cite{Kimoto--Wakayama}, (see also \cite{Pollicott--Sharp,Sarnak--Wakayama}).
	
	We end up our introduction by sketching a solution to our example task,
	to construct of a good curve $\gamma$ with a null homologous canonical lift $\hat{\gamma}$.
	As mentioned, 
	the first attempt is to create a null homologous good curve $\zeta$.
	We take two long closed geodesic paths $a,b$ based at the same point $p\in M$,
	making sure that they point very sharply against each other, with little twist
	in the normal direction.
	In other words, we apply the connection principle as usual to
	construct $a$ and $b$ so that the initial direction of $a$ and
	the terminal direction of $b$ are approximately some unit vector $\vec{t}_p$ at $p$,
	while the terminal direction of $a$ and the initial direction of $b$ are approximately $-\vec{t}_p$;
	moreover, we require the parallel transport of some unit vector $\vec{n}_p\perp\vec{t}_p$
	at $p$ along either $a$ or $b$ is approximately $\vec{n}_p$.
	Then a null homologous good curve $\zeta$ 
	can be constructed as the (reduced) cyclic concatenation of
	the commutator word $ab\bar{a}\bar{b}$. 
	However, a fundamental calculation tells us that
	the canonical lift $\hat{\zeta}$ represents
	the nontrivial element $[\hat{c}]\in H_1(\SO(M);\Integral)$,
	in general, (see Lemma \ref{estimateCommutator}).
	This should not be too surprising because of the following intuition:
	If we were on an oriented closed hyperbolic surface rather than in a $3$-manifold,
	the same construction would	work perfectly well, 
	but there would be no chance for $\zeta$ to bound an immersed subsurface,
	since it has an odd self-intersection number.
		
	We are therefore hinted to make some essential use of the extra dimension. 
	The trick is to intentionally misalign the invariant normal vectors of $a$ and $b$
	with some small angular difference, and take a suitable power of $\zeta$ to be the desired $\gamma$.
	For simplicity, let us assume that the injectivity radius of $M$ is not too small, say, at least $1$.
	Choose two unit normal vectors $\vec{n}_p$ and 
	$\vec{n}_p(2\pi/400)$, the rotation of $\vec{n}_p$ about $\vec{t}_p$ of the angle $2\pi/400$.
	We construct the closed paths $a$ and $b$ as before except asking
	$a$ to preserve $\vec{n}_p$ and $b$ to preserve $\vec{n}_p(2\pi/400)$, approximately.
	This causes some tiny nontrivial honolomy of the concatenated path
	$ab\bar{a}\bar{b}$: The transport of $\vec{n}_p$ along the path
	becomes approximately $\vec{n}_p(2\pi/100)$.
	If we believe, for the intuitive reason above,
	that the parallel transport path of the frame 
	$\mathbf{p}=(\vec{t}_p,\vec{n}_p,\vec{t}_p\times\vec{n}_p)$ 
	along the commutator path is approximately relatively homologous
	to the path of spinning frames $\mathbf{p}(\phi)$ about $\vec{t}_p$ 
	parametrized by the angle $\phi\in [0,2\pi/100]$, 
	then we can expect that the cyclically concatenated path $(ab\bar{a}\bar{b})^{100}$
	gives rise to a null homologous good curve $\gamma$ with a null homologous canonical lift $\hat{\gamma}$,
	as if the accumulated effect of the normal spinning corrected
	the homology class of $\hat\gamma$ from $[\hat{c}]$ back to $0$.
	To implement the idea, of course, we need to set up fundamental calculations and estimates,
	and make careful choice of error scales and other constants. 
	The idea condenses roughly to Proposition \ref{substituteFrameSpinningPath}, 
	which is the heart of Theorem \ref{gradedConnectionPrinciple}.
	
	The paper is organized as follows.
	In Section \ref{Sec-preliminaries}, we review the good pants construction of Kahn--Markovic and
	its subsequent development. In Section \ref{Sec-connectingFrames}, we introduce our enhanced version of 
	the connection principle. The proof of Theorem \ref{constructGradedGoodCurves} can be found at the end of that
	section. Section \ref{Sec-oddEulerCharacteristic} contains the proof of our theme result, Theorem \ref{main-oddEulerCharacteristic}.
	The essential construction for Theorem \ref{exponentialTorsionGrowth} is presented in Section \ref{Sec-irregularExhaustingTower},
	and the complete proof is summarized in Section \ref{Sec-uniformLattices}.
	
	\subsection*{Acknowledgements} 
	The author would like to thank Ian Agol and Vlad Markovic for many interesting and
	inspiring conversations. The author also thanks Hongbin Sun for valuable communications.

\section{Preliminaries}\label{Sec-preliminaries}
		This section contains a compact introduction to
		the good pants method. This method has been invented by J.~Kahn and V.~Markovic \cite{KM-surfaceSubgroup}
		to resolve the Surface Subgroup Conjecture
		and developed by various authors in subsequent works,
		see \cite{Hamenstadt,KM-Ehrenpreis,Liu-Markovic,Saric,Sun-virtualTorsion,Sun-virtualDomination}.
		For closed hyperbolic $3$-manifolds	and potentially for other compact rank-one locally symmetric spaces,
		it provides a package of tools which enables one to conveniently construct certain
		$\pi_1$--injectively	immersed $2$-subcomplexes,
		 especially nearly totally geodesic subsurfaces.
		For the purpose of this paper, 
		we restrict our discussion to closed hyperbolic $3$-manifolds 
		and include only a minimal collection of relevant materials. 
		
		Let $M$ be an oriented closed hyperbolic $3$-manifold.
		The idea of the good pants method is to produce
		$\pi_1$--injectively immersed $2$-complexes in $M$
		by gluing pairs of \emph{good pants} in a nice way
		along their common cuffs, which are \emph{good curves}.
		Roughly speaking, a good curve in $M$ is an oriented geodesic loop with nearly trivial holonomy; 
		a pair of good pants in $M$ is an oriented immersed	pair of pants 
		which is nearly regular and nearly totally geodesic.
		
		In quantitative terms, suppose that $0<\epsilon<1$ and $R>1$ are given constants,
		where $\epsilon$ is presumably very small and $R$ very large.
		An \emph{$(R,\epsilon)$--curve}
		is defined to be an oriented geodesic loop 
		of complex length approximately the real value $R$,
		with the error required to be at most $\epsilon$ in absolute value of the difference.
		A pair of \emph{$(R,\epsilon)$--pants}
		is defined to be an (unmarked) oriented immersed pair of pants
		with the following extra requirements on its shape:
		The three cuffs are required to be $(R,\epsilon)$--curves; 
		the three seams are required to be the unique arcs
		of the pair of pants 
		which connect between the cuffs as common perpendicular geodesic arcs;
		as hence there are six well-defined feet, namely,
		the unit normal vectors of the cuffs attached at the endpoints of the seams pointing into the seams,
		the final requirement is that each foot of a cuff 
		should be approximately equal to the parallel transport of the other foot of that cuff,
		along the half cuff in the forward direction,
		and that the error should be at most $\epsilon$	measured in angle between unit normal vectors.
		In other words, it only makes sense to speak of the complex \emph{half} length
		of an $(R,\epsilon)$--curve as a cuff of a (nonsingular) $\pi_1$--injectively immersed pair of pants,
		and being a pair of good pants requires all the cuffs to have complex half length $\epsilon$--close to $R/2$ 
		(rather than $R/2+\pi\cdot\sqrt{-1}$).
		We do not distinguish $(R,\epsilon)$--curves which are the same up to change of parametrization,
		or $(R,\epsilon)$-pants which are the same up to homotopy and orientation-preserving self-homeomorphism.
		The reader is referred to \cite[Section 2]{KM-surfaceSubgroup} for the original definition 
		and \cite[Section 2]{Liu-Markovic} for expanded discussions.
		
		Adopt the notations
		$$\ocurves_{R,\epsilon}(M)=\{\,(R,\epsilon)\textrm{--curves of }M\};\
			\opants_{R,\epsilon}(M)=\{\,(R,\epsilon)\textrm{--pants of }M\,\}.$$
		For any given constant $\epsilon$, as long as $R$ is sufficiently large,
		where an \textit{a priori} lower bound depends only on $M$ and $\epsilon$,
		$\ocurves_{R,\epsilon}(M)$	and $\opants_{R,\epsilon}(M)$ are always non-empty. 
		In fact, $(R,\epsilon)$--curves and $(R,\epsilon)$--pants can be produced 
		using the following basic construction, 
		which is an immediate consequence of the (exponential) mixing property of the frame flow,
		(see \cite[Lemma 4.4]{KM-surfaceSubgroup}, \cite[Lemma 4.15]{Liu-Markovic}):
		
		\begin{proposition}[Connection principle]\label{connectionPrinciple}
			Let $M$ be an oriented closed hyperbolic $3$-manifold.
			%
			Let	$\vec{t}_p\perp\vec{n}_p$ and $\vec{t}_q\perp\vec{n}_q$
			be two pairs of orthogonal unit vectors	at the points $p$ and $q$ of $M$, respectively.
			%
			
			Given any positive constant $\delta$,
			and for every sufficiently large positive constant $L$ with respect to
			$M$ and $\delta$, 
			there exists a geodesic path $s$ in $M$ from $p$ to $q$ with the following properties:
			\begin{itemize}
				\item The length of $s$ is $\delta$--close to $L$. The initial direction of $s$ is $\delta$--close 
				to $\vec{t}_p$ and the terminal direction of $s$ is $\delta$--close 
				to $\vec{t}_q$, where the distance is measured by the angle between unit vectors.
				\item The parallel transport from $p$ to $q$ along $s$ takes $\vec{n}_p$ to a unit vector
				$\vec{n}'_q$ which is $\delta$--close to $\vec{n}_q$.
			\end{itemize}
		\end{proposition}
		
		Moreover,	as the required cuff length $R$ grows, there will eventually be
		plenty of $(R,\epsilon)$--pants and their feet tend to be very evenly distributed
		on the unit normal bundle over every $(R,\epsilon)$--curve. 
		By nicely gluing a suitable collection of good pants along common cuffs, Kahn and Markovic are able to resolve the Surface Subgroup Conjecture
		for closed hyperbolic $3$-manifolds \cite{KM-surfaceSubgroup}:
		
		\begin{theorem}\label{surfaceSubgroupTheorem}
			Every closed hyperbolic $3$-manifold contains a $\pi_1$--injectively immersed quasi-Fuchsian closed subsurface.
		\end{theorem}
		
		Associated with the constructed surface
		is a naturally induced pants decomposition with markings,
		so the shape of surface can be described 
		by its complex Fenchel--Nielsen coordinates.
		In these terms, the complex half length and the shearing parameter for
		each glued cuff $C$ are approximately $(\chlen(C),s(C))\approx (R/2,1)$.
		The first component has error at most $\epsilon$ and the second component has error at most $\epsilon/R$. 
		%
		%
		
		The relative version of the construction problem has been studied in \cite{Liu-Markovic},
		namely, whether a collection of good curves can bound
		a subsurface which is nicely glued from good pants.
		Regarding any given collection of $(R,\epsilon)$--curves (possibly with multiplicity) as
		an element of the free integral module $\Integral\ocurves_{R,\epsilon}(M)$ 
		(in the non-negative hyper-octant),
		and similarly $(R,\epsilon)$--pants as of $\Integral\opants_{R,\epsilon}(M)$,
		an obstruction to solving the relative problem
		lies in the cokernel of the homomorphism:
			$$\partial\colon \Integral\opants_{R,\epsilon}(M)\longrightarrow\Integral\ocurves_{R,\epsilon}(M),$$
		defined by taking any pair of pants	to the sum of its three cuffs.
		For large $R$ with respect to $\epsilon$ and $M$, 
		the \emph{$(R,\epsilon)$-panted cobordism group} of $M$ is defined to be:
		$$\ocobordism_{R,\epsilon}(M)\,=\,\Integral\ocurves_{R,\epsilon}(M)\,/\,\partial(\Integral\opants_{R,\epsilon}(M)).$$
		The group $\ocobordism_{R,\epsilon}(M)$ can be fully characterized by the following correspondence \cite[Theorem 5.2]{Liu-Markovic}.
		We denote by $\SO(M)$ the $\SO(3)$--principal bundle over $M$ of orthonormal frames with right orientation.
				
		\begin{theorem}\label{goodPantedCobordismGroupLM}
			For any sufficiently small positive constant $\epsilon$ with respect to $M$, 
			and any sufficiently large positive constant $R$ with respect to $M$ and $\epsilon$,
			there exists a canonical isomorphism:
				$$\Phi\colon \ocobordism_{R,\epsilon}(M)\longrightarrow H_1(\SO(M);\,\Integral).$$
			Moreover, for all $\gamma\in\ocurves_{R,\epsilon}(M)$,
			the projection of $\Phi(\gamma)$ in $H_1(M;\,\Integral)$ equals the homology class of $\gamma$.
		\end{theorem}
		
		For the goal of this paper, it is important to understand 
		the implementation of $\Phi$. 
		It suffices to specify the assignment of $\Phi$ to each curve $\gamma\in\ocurves_{R,\epsilon}(M)$.
		In fact, the homology class $\Phi(\gamma)$ can be 
		represented by the \emph{canonical lift} of the curve $\gamma$,
		denoted as
			$$\hat\gamma\in\pi_1(\SO(M)),$$
		and $\hat{\gamma}$ is defined as follows.
				
		Suppose that $\gamma$ is an $(R,\epsilon)$--curve. Take a special orthonormal frame 
		$\mathbf{p}=(\vec{t}_p,\vec{n}_p,\vec{t}_p\times\vec{n}_p)\in\SO(M)$
		at a point $p\in\gamma$ such that $\vec{t}_p$ is the unit tangent vector of $\gamma$ at $p$.
		Here the cross notation stands for the cross product for the oriented tangent space $T_pM$ with
		the Riemannian inner product.
		For sufficiently large $R$ with respect to $M$ and $\epsilon$, the parallel transport
		of $\mathbf{p}$ along $\gamma$ back to $p$ is a frame $\mathbf{p}'$ which can be connected
		with $\mathbf{p}$ by a unique shortest path in $\SO(M)|_p$.
		Then the canonical lift $\hat\gamma$ can be represented by the closed path 
		which is the concatenation of three	consecutive subpaths: The first is the parallel-transportation path 
		from $\mathbf{p}$ to $\mathbf{p}'$ along $\gamma$; the second is a closed path in $\SO(M)|_p$ based at
		$\mathbf{p}'$ which represents the unique nontrivial central element
			$$\hat{c}\in\pi_1(\SO(M));$$
		and the third is the shortest path in $\SO(M)|_p$ connecting $\mathbf{p}'$ and $\mathbf{p}$.
		The resulting loop of frames does not depend on 
		the auxiliary choices	up to free homotopy in $\SO(M)$, so $\hat\gamma\in\pi_1(\SO(M))$ is well defined.
		Note that the second sub-path above ensures that $\Phi$ vanishes on boundary of pants.
				
		It turns out that elements of $\ocobordism_{R,\epsilon}(M)$ 
		are the only obstructions to solving the relative construction problem.
		For simplicity, we state a special case as follows. 
		The case with multicurve boundary is completely analogous, 
		see \cite{Liu-Markovic}, \cite[Corollary 2.7]{Sun-virtualDomination}.
				
		\begin{theorem}\label{solutionRelative}
			For any sufficiently small positive constant $\epsilon$ with respect to $M$, 
			and any sufficiently large positive constant $R$ with respect to $M$ and $\epsilon$,
			the following statement holds true. 
			
			For any curve $\gamma\in\ocurves_{R,\epsilon}(M)$, 
			there exists an oriented, connected, $\pi_1$--injectively immersed
			quasi--Fuchsian subsurface of $M$ which is $(R,\epsilon)$--panted and bounded by $\gamma$
			if and only if the canonical lift $\hat\gamma\in\pi_1(M)$ is null homologous in $\SO(M)$.
		\end{theorem}
		
		Note that Theorem \ref{solutionRelative} does not say anything about the
		existence of $(R,\epsilon)$-curves
		with homologically trivial canonical lifts.
		The existence follows from Theorem \ref{constructGradedGoodCurves},
		which is proved in the rest of this paper.
		
		We finish this section with the following finer version of the good pants method
		that we have discussed so far. 
		This has been pointed out by \cite[Remark 2.9]{Sun-virtualDomination},
		see also \cite{Saric}, as a consequence of the exponential mixing rate of the frame flow. 
		In the rest of this paper, we adopt the finer version
		so as to invoke some estimations from \cite{Sun-virtualTorsion}.
		
		\begin{declaration}\label{finerError}
		All the notations and results that have appeared in this section remain compatible and true with
		the constant $\epsilon$ replaced by $\epsilon/R$, (and $\delta$ replaced by $\delta/L$).
		\end{declaration}

	\section{Connecting frames in a prescribed homology class}\label{Sec-connectingFrames}
		
		Let $M$ be an oriented closed hyperbolic $3$-manifold. Denote by $\SO(M)$ the special orthonormal frame bundle over $M$.
		Let	$\mathbf{p}$ and $\mathbf{q}$
		be a pair of (possibly coincident) frames in $\SO(M)$
		at points $p$ and $q$ of $M$, respectively.
		Denote by
			$$\pi_1(\SO(M),\mathbf{p},\mathbf{q})$$
		the set of relative homotopy classes of paths in $\SO(M)$ from $\mathbf{p}$ to $\mathbf{q}$.
		Each element $\hat{\xi}$ of $\pi_1(\SO(M),\mathbf{p},\mathbf{q})$ represents a relative homology class
		$[\hat\xi]\in H_1(\SO(M),\mathbf{p}\cup\mathbf{q};\,\Integral)$ with the property
		$\partial_*[\hat\xi]=[\mathbf{q}]-[\mathbf{p}]$, where
			$$\partial_*\colon H_1(\SO(M),\mathbf{p}\cup\mathbf{q};\,\Integral)\longrightarrow H_0(\mathbf{p}\cup\mathbf{q};\,\Integral)$$
		is the induced homomorphism of the relative pair $(\SO(M),\mathbf{p}\cup\mathbf{q})$.
		Since the preimage $\partial_*^{-1}([\mathbf{q}]-[\mathbf{p}])$ is an affine $H_1(\SO(M);\Integral)$,
		one may think of the assignment $\hat\xi\mapsto[\hat\xi]$ as a model of 
		a natural relative grading for $\pi_1(\SO(M),\mathbf{p},\mathbf{q})$ such that the grading differences between elements
		are valued in $H_1(\SO(M);\,\Integral)$.
		In this way, $\pi_1(\SO(M),\mathbf{p},\mathbf{q})$ decomposes naturally into the disjoint union
		of its grading classes.
		
		The following theorem is an enhanced version of the connection principle which respects gradings.
								
		\begin{theorem}\label{gradedConnectionPrinciple}
			Let $M$ be an oriented closed hyperbolic $3$-manifold.
			Let	$\mathbf{p}=(\vec{t}_p,\vec{n}_p,\vec{t}_p\times\vec{n}_p)$ and $\mathbf{q}=(\vec{t}_q,\vec{n}_q,\vec{t}_q\times\vec{n}_q)$
			in $\SO(M)$ be a pair of special orthonormal frames
			at points $p$ and $q$ of $M$, respectively.
			Let $\Xi\in H_1(\SO(M),\mathbf{p}\cup\mathbf{q};\,\Integral)$ be a relative
			homology class with boundary $\partial_*[\hat\xi]=[\mathbf{q}]-[\mathbf{p}]$.
			
			Given any positive constant $\delta$,
			and for every sufficiently large positive constant $L$ with respect to
			$M$, $\Xi$, and $\delta$, there exists a geodesic path $s$ in $M$ from $p$ to $q$ with the following properties:
			\begin{itemize}
				\item The length of $s$ is $(\delta/L)$--close to $L$. The initial direction of $s$ is $(\delta/L)$--close 
				to $\vec{t}_p$ and the terminal direction of $s$ is $(\delta/L)$--close 
				to $\vec{t}_q$.
				\item The parallel transport from $p$ to $q$ along $s$ takes $\mathbf{p}$ to a frame
				$\mathbf{q}'$ which is $(\delta/L)$--close to $\mathbf{q}$, and there exists a unique shortest
				path in $\SO(M)|_q$ between $\mathbf{q}'$ and $\mathbf{q}$. 
				\item Denote by $\hat{s}$ the path which is the concatenation of the parallel-transport
				path from $\mathbf{p}$ to $\mathbf{q}'$
				with the shortest path from $\mathbf{q}'$ to $\mathbf{q}$.
				The relative homology class represented by $\hat{s}\in\pi_1(\SO(M),\mathbf{p},\mathbf{q})$ equals $\Xi$.
			\end{itemize}		
		\end{theorem}
		
		Note that when $\mathbf{p}$ coincides with $\mathbf{q}$, the union $\mathbf{p}\cup\mathbf{q}$
		is understood as a single point in $\SO(M)$.

		In the rest of this section, we prove Theorem \ref{gradedConnectionPrinciple} by construction.
		We fix an oriented closed hyperbolic $3$-manifold $M$ throughout this section.
				
		\subsection{Basic calculations for parallel-transport paths of frames}
		We start by three calculations for the homology classes of frame paths.
		The first calculation tells us that
		if we concatenate a consecutive chain 
		of long geodesic paths in $M$ with small total bending,
		then parallel transport of any frame along the concatenated path or its reduction yields 
		almost the same path of frames up to homotopy.
		The second calculation is parallel to the first one, for the reduction of
		the cyclic concatenation of a consecutive cycle.
		The third calculation considers a special situation that  
		if we concatenate a consecutive commutator chain of 
		long geodesic paths in $M$ with small bending and twisting,
		then parallel transport of any frame along the concatenated path 
		yields a path of frames which is almost closed and null homologous in the frame bundle.
		
		By a \emph{consecutive chain} of geodesic paths in $M$, we mean a finite sequence of oriented geodesic paths
		such that the terminal endpoint of any member of the sequence, except the final one,
		is the initial endpoint of its successor. The \emph{reduced concatenation} of a consecutive chain
		is the unique geodesic path in $M$ 
		which is homotopic to the concatenation of the chain relative to the endpoints.
		A \emph{consecutive cycle} of geodesic paths is a chain whose last member
		has its terminal endpoint the same as the initial point of the first one.
		The \emph{reduced cyclic concatenation} is the unique geodesic loop without base point
		which is freely homotopic to the cyclic concatenation of the cycle.
		
		We do not attempt to make the most economic choices 
		for universal constants	in the estimates. 
		Instead, powers of ten are often used,
		and the power roughly counts 
		the number of basic steps that have been taken.
		
		\begin{lemma}\label{estimateReduction}
			Let $s_1,\cdots,s_m$ be a consecutive chain of $m$ geodesic paths in $M$. Suppose that the length of each $s_i$
			is at least $L$, and suppose that the terminal direction of each $s_i$, except $s_m$, is $\delta$--close 
			to the initial direction of $s_{i+1}$. Denote by $s$ the reduced concatenation of $s_1,\cdots,s_m$.
			Let $\mathbf{p}$ be a frame in $\SO(M)$ at the initial endpoint $p$ of $s_1$.
			Denote by $\mathbf{q}$ and $\mathbf{q'}$
			the parallel transport of $\mathbf{p}$
			along the concatenation of $s_1,\cdots,s_m$ 
			and along $s$ to the terminal endpoint $q$ of $s_m$, respectively.
			
			Then for universally small $(m\delta)$ and sufficiently large $L$ with respect to $\delta$, 
			there exists a unique shortest path in $\SO(M)|_q$ from $\mathbf{q}'$ to $\mathbf{q}$. 
			Moreover, the parallel-transport path from $\mathbf{p}$ to $\mathbf{q}$ is homotopic to
			the concatenation of the parallel-transport path from $\mathbf{p}$ to $\mathbf{q}'$
			and the shortest path in $\SO(M)|_q$ from $\mathbf{q}'$ to $\mathbf{q}$.
		\end{lemma}
		
		\begin{proof}
			First consider the basic case if $m$ equals $2$. Then the consecutive geodesic paths $s_1$, $s_2$
			spans a unique geodesic $2$-simplex $T$, which is immersed in $M$,
			and the reduced concatenated path $s$ is geometrically the third edge of $T$.
			For universally small $\delta$ and sufficiently large $L$ with respect to $\delta$,
			the area of $T$ can be bounded by
			$2\delta$. As parallel transport takes any frame $\mathbf{p}\in\SO(M)|_p$
			along	the chain $s_1,s_2$ to $\mathbf{q}\in\SO(M)|_q$, and along $s$ to $\mathbf{q}'\in\SO(M)|_q$,
			it follows that the distance between $\mathbf{q}$ and $\mathbf{q}'$ can be bounded by $10\delta$.
			We may in addition require that $\delta$ is so small that 
			that any ball of radius $100\delta$ in $\SO(M)|_q$ is embedded.
			Then there is a unique path of the shortest length in $\SO(M)|_q$ connecting $\mathbf{q}'$ and $\mathbf{q}$,
			which we denote as $\hat{\eta}$.
			
			To see the claimed homotopy in this case,
			let $h\colon [0,1]\to T$ be the altitude of $T$ on the side $s$ 
			so that $h(1)$ is the joint point of $s_1$ and $s_2$,
			and $h(0)$ lies on $s$. For a parameter $t\in [0,1]$,
			consider the $t$-family of consecutive chains $s_1(h(t))$, $s_2(h(t))$ in $T$
			which is formed by the two geodesic segments $[p,h(t)]$ and $[h(t),q]$.
			Denote by $\hat{\xi}_t$ the path of frames 
			given by the parallel transport of $\mathbf{p}$ along 
			the concatenation of $s_1(h(t)),s_2(h(t))$. 
			As $t\in[0,1]$ varies, the endpoint of $\hat{\xi}_t$ in $\SO(M)|_q$ 
			gives rise to a path $\hat{\eta}'$ from $\mathbf{q}'$ to $\mathbf{q}$.
			It follows from the construction of $\hat{\xi}_t$ that $\hat{\xi}_0\hat{\eta}'$ is homotopic
			to $\hat{\xi}_1$, the parallel-transport path along the concatenated chain $s_1,s_2$.
			On the other hand, the distance estimation as above can also be applied for pairs of points
			on $\hat{\eta}'$, so $\hat{\eta}$ has diameter bounded by $10\delta$ in $\SO(M)|_q$. It follows
			that $\hat{\eta}'$ can be homotoped to $\hat{\eta}$ within $\SO(M)|_q$.
			Then $\hat{\xi}_1$ is homotopic to $\hat{\xi}_0\hat{\eta}$, which proves the basic case.
			
			The general case can be done by applying the basic case for $m-1$ times. In other words,
			we triangulate the polygon formed by the consecutive chain $s_1,\cdots,s_m$ and the reduced concatenation
			$s$ by adding diagonals, and then apply the basic case to each geodesic $2$-simplex that fills a triangle.
			(See \cite[Lemma 4.8 (a)]{Liu-Markovic} for the geometry of the diagonals.)
			Then the total area of the $2$-simplices is bounded by about $m\delta$.
			When this quantity
			is small enough, $\mathbf{q}'$ and $\mathbf{q}$ can be connected by a uniquely shortest path in $\SO(M)|_q$,
			of length at most $100m\delta$. We omit the details because the estimation is straightforward.
		\end{proof}
		
		\begin{lemma}\label{estimateCyclicReduction}
			Let $s_1,\cdots,s_m$ be a consecutive cycle of $m$ geodesic paths in $M$. 
			Suppose that the length of each $s_i$
			is at least $L$, and suppose that the terminal direction of each $s_i$ is $\delta$--close 
			to the initial direction of $s_{i+1}$, where $s_{m+1}$ means $s_1$.
			Denote by $\gamma$ the reduced cyclic concatenation of $s_1,\cdots,s_m$.
			Let $\mathbf{p}$ be a frame in $\SO(M)$ at the initial endpoint $p$ of $s_1$.
			Denote by $\mathbf{p}'$ the parallel transport of $\mathbf{p}$
			along the cyclic concatenation of $s_1,\cdots,s_m$ back to $p$. 
			Let $\mathbf{q}$ be a frame in $\SO(M)$ at a point $q'\in\gamma$.
			Denote by $\mathbf{q}'$ the parallel transport of $\mathbf{q}$
			along $\gamma$ back to $q$.
			
			Then for universally small $(m\delta)$ and sufficiently large $L$ with respect to $\delta$, 
			there exist unique shortest paths in $\SO(M)|_p$ from $\mathbf{p}'$ to $\mathbf{p}$,
			and in $\SO(M)|_q$ from $\mathbf{q}'$ to $\mathbf{q}$.
			Moreover, the cyclic concatenation of 
			parallel-transport path from $\mathbf{p}$ to $\mathbf{p}'$ with 
			the shortest path in $\SO(M)|_p$ from $\mathbf{p}'$ to $\mathbf{p}$
			is freely homotopic to
			the cyclic concatenation of
			the parallel-transport path from $\mathbf{q}$ to $\mathbf{q}'$ with
			and the shortest path in $\SO(M)|_q$ from $\mathbf{q}'$ to $\mathbf{q}$.
		\end{lemma}
		
		\begin{proof}
			Note that under the assumptions,
			the annulus between the cyclic concatenation of the cycle and its reduction
			can be chosen to have small area, and $q\in\gamma$ can be chosen $10\delta$--close to $p$.
			See \cite[Lemma 4.8 (b)]{Liu-Markovic} for the geometry.
			Then the lemma can be proved in a way similar to Lemma \ref{estimateReduction}.
		\end{proof}

		\begin{lemma}\label{estimateCommutator}
			Let $a,b$ be a pair of geodesic paths in $M$ with all the endpoints the same point $p$.
			Suppose that there exists a unit vector $\vec{t}$ at $\vec{p}$ such that
			$\vec{t}$ is $\delta$--close to the initial direction of $a$ and the terminal direction of $b$,
			and that $-\vec{t}$ is $\delta$--close to the terminal direction of $a$ and the initial direction of $b$.
			Moreover, suppose that there exists a unit vector $\vec{n}\perp\vec{t}$ at $p$
			of which the parallel transport along $a$ and along $b$ are both $\delta$--close to $\vec{n}$.
			Let $\mathbf{p}$ be a frame in $\SO(M)$ at $p$.
			Denote by $\mathbf{q}$ the parallel transport of $\mathbf{p}$
			along the concatenated path $ab\bar{a}\bar{b}$, where the bar notation stands for the orientation reversal.
			
			Then for universally small $\delta$, 
			there exists a unique shortest path in $\SO(M)|_p$ from $\mathbf{q}$ to $\mathbf{p}$. 
			Moreover, the concatenation of 
			the parallel-transport path from $\mathbf{p}$ to $\mathbf{q}$ and the shortest path
			in $\SO(M)|_p$ from $\mathbf{q}$ to $\mathbf{p}$ is null homologous in $\SO(M)$.
		\end{lemma}
		
		\begin{proof}
			Observe that it suffices to prove the lemma for a specific frame $\mathbf{p}$ in $\SO(M)|_p$.
			In fact, suppose that we have found a	(cellular) $2$-chain 
			 $C=k_1\sigma_1+\cdots+k_r\sigma_r$ bounded by the claimed cycle for $\mathbf{p}$,
			where $\sigma_i\colon D^2\to \SO(M)$ are $2$-cells.
			Since any other frame $\mathbf{p}'$ in $\SO(M)|_p$
			can be written as $\mathbf{p}\cdot g$ where $g\in\SO(3)$
			is a constant matrix, the $2$-chain $C\cdot g=k_1\sigma_1\cdot g+\cdots+k_r\sigma_r\cdot g$
			is bounded by the claimed cycle for $\mathbf{p}'$.
			This is because the right action of $\SO(3)$ on the principal bundle $\SO(M)$
			preserves the metric and commutes with parallel transport.
			
			In the rest of the proof, we argue for the frame $\mathbf{p}=(\vec{t},\vec{n},\vec{t}\times\vec{n})$
			in $\SO(M)|_p$. Denote by $\mathbf{p}^\dagger\in\SO(M)$ the frame $(-\vec{t},\vec{n},-\vec{t}\times\vec{n})$.
			
			For universally small $\delta$, we may assume
			that any ball of radius $1000\delta$ in $\SO(M)|_p$ is embedded. 
			By the assumption, the parallel transport of
			$\mathbf{p}$ consequentially along $a$, $b$, $\bar{a}$, $\bar{b}$ gives rise to
			four frames $\mathbf{p}_1,\cdots,\mathbf{p}_4\in\SO(M)|_p$ such that
			$\mathbf{p}_1$ and $\mathbf{p}_3$ lie in the $(100\delta)$--neighborhood of $\mathbf{p}^\dagger$,
			while $\mathbf{p}_2$ and $\mathbf{p}_4$ lie in the $(100\delta)$--neighborhood of $\mathbf{p}$.
			Note that $\mathbf{p}_4=\mathbf{q}$, and
			it is already clear that $\mathbf{q}$ can be connected by a unique shortest path in $\SO(M)|_p$.
						
			It remains to show that the claimed cycle is a boundary in $\SO(M)$. Writing $\mathbf{p}_0=\mathbf{p}$, 
			the claimed cycle as the sum of the parallel-transport paths 
			$[\mathbf{p}_i,\mathbf{p}_{i+1}]$ for $i=0,1,2,3$ and the shortest path $[\mathbf{p}_4,\mathbf{p}_0]$ in $\SO(M)|_p$.
			We argue by homologically simplifying the claimed cycle until it is obviously null homologous in $\SO(M)$.
			To this end, denote by
				$$\mathbf{p}\cdot R_N(\theta)\in\SO(M)|_p$$
			the rotation of $\mathbf{p}$ about $\vec{n}$ by an angle $\theta$,
			which is given by the right action of the matrix
			$$R_N(\theta)\,=\,\left[
			\begin{array}{ccc} 
			\cos\theta&0&\sin\theta\\
			0&1&0\\
			-\sin\theta&0&\cos\theta
			\end{array}
			\right]\in\SO(3).$$
			Note that $\mathbf{p}\cdot R_N(0)=\mathbf{p}$ and $\mathbf{p}\cdot R_N(\pi)=\mathbf{p}^\dagger$.
			Denote by $\mathbf{p}_i^\dagger=\mathbf{p}_i\cdot R_N(\pi)$,
			and by $\hat{\xi}_i$ the path from $\mathbf{p}_i$ to $\mathbf{p}_i^\dagger$,
			parametrized as $\mathbf{p}\cdot R_N(\theta)$ for $\theta\in[0,\pi]$.
			We have approximately
			$$\mathbf{p}_0,\mathbf{p}_1^\dagger,\mathbf{p}_2,\mathbf{p}_3^\dagger,\mathbf{p}_4\approx\mathbf{p}$$
			with error at most $100\delta$ from $\mathbf{p}$,
			and similarly,
			$$\mathbf{p}_0^\dagger,\mathbf{p}_1,\mathbf{p}_2^\dagger,\mathbf{p}_3,\mathbf{p}_4^\dagger\approx\mathbf{p}^\dagger.$$
			By our assumption on the smallness of $\delta$,
			those frames near $\mathbf{p}$, or $\mathbf{p}^\dagger$ respectively,
			can be mutually connected by unique shortest paths in $\SO(M)|_p$. 
			
			Consider the paths $[\mathbf{p}_0,\mathbf{p}_1]$ and $[\mathbf{p}_2,\mathbf{p}_3]$.
			Because there is a rectangle parametrized as a family of parallel-transport paths
			$[\mathbf{p}_2,\mathbf{p}_3]\cdot R_N(\theta)$ along $\bar{a}$ where $\theta\in[0,\pi]$,
			the path $[\mathbf{p}_2,\mathbf{p}_3]$ is homologous to the ($1$-chain) 
			sum of $\hat{\xi}_2$ and $[\mathbf{p}_2,\mathbf{p}_3]\cdot R_N(\pi)$ and $-\hat{\xi}_3$.
			The middle term $[\mathbf{p}_2,\mathbf{p}_3]\cdot R_N(\pi)$ equals the parallel-transport
			path from $\mathbf{p}_2^\dagger$ along $\bar{a}$, where $\mathbf{p}_2^\dagger$
			is approximately $\mathbf{p}^\dagger$. Since $[\mathbf{p}_0,\mathbf{p}_1]$ is the parallel-transport
			path from $\mathbf{p}_0$ along $a$, where $\mathbf{p}_0$ is approximately $\mathbf{p}$,
			the path $[\mathbf{p}_2,\mathbf{p}_3]\cdot R_N(\pi)$ is almost the orientation reversal
			of $[\mathbf{p}_0,\mathbf{p}_1]$. More precisely, their sum is homologous to the sum of the shortest paths 
			$[\mathbf{p}_0,\mathbf{p}_2^\dagger]$ and 
			$[\mathbf{p}_3^\dagger,\mathbf{p}_1]$ near $\mathbf{p}$ and $\mathbf{p}^\dagger$, respectively.
			This yields
				$$[\mathbf{p}_0,\mathbf{p}_1]+[\mathbf{p}_2,\mathbf{p}_3]
				=\hat\xi_2-\hat\xi_3+[\mathbf{p}_0,\mathbf{p}_2^\dagger]+[\mathbf{p}_3^\dagger,\mathbf{p}_1],$$
			as $1$-chains of $\SO(M)$ modulo $1$-boundaries.
			Similarly,
				$$[\mathbf{p}_1,\mathbf{p}_2]+[\mathbf{p}_3,\mathbf{p}_4]
				=\hat\xi_3-\hat\xi_4+[\mathbf{p}_1,\mathbf{p}_3^\dagger]+[\mathbf{p}_4^\dagger,\mathbf{p}_2],$$
			modulo $1$-boundaries.
			Since $\mathbf{p}_2$ and $\mathbf{p}_4$ are near to each other,
				$$\hat\xi_2-\hat\xi_4=-[\mathbf{p}_2^\dagger,\mathbf{p}_4^\dagger]-[\mathbf{p}_4,\mathbf{p}_2],$$
			modulo $1$-boundaries.
			Therefore, the claimed cycle
				$$[\mathbf{p}_0,\mathbf{p}_1]+[\mathbf{p}_1,\mathbf{p}_2]+[\mathbf{p}_2,\mathbf{p}_3]+[\mathbf{p}_3,\mathbf{p}_4]+[\mathbf{p}_4,\mathbf{p}_0]$$
			is equal to a linear combination of shortest paths in $\SO(M)|_p$ 
			that are near $\mathbf{p}$ or $\mathbf{p}^\dagger$,
			modulo $1$-boundaries of $\SO(M)$. 
			It follows that the claimed cycle is null homologous in $\SO(M)$.			
		\end{proof}
		
		\subsection{Substitution for paths of spinning frames}
		
		Given a parameter $\phi\in\Real$, and for any frame 
		$\mathbf{p}=(\vec{t}_p, \vec{n}_p, \vec{t}_p\times\vec{n}_p)$ in $\SO(M)$ at a point $p$ in $M$,
		we denote by
		$$\mathbf{p}(\phi)\,=\,(\vec{t}_p,\vec{n}_p(\phi),\vec{t}_p\times\vec{n}_p(\phi))\in\SO(M)|_p$$
		the rotation of $\mathbf{p}$ about $\vec{t}_p$ by an angle $\phi$.
		In other words,
		$$\mathbf{p}(\phi)\,=\,\mathbf{p}\cdot R_T(\phi),$$
		where
		$R_T(\phi)\in\SO(3)$ stands for the matrix
		$$R_T(\phi)\,=\,\left[
			\begin{array}{ccc} 
			1&0&0\\
			0&\cos\phi&-\sin\phi\\
			0&\sin\phi&\cos\phi
			\end{array}
			\right].$$
		This gives rise to a 1-parameter family of spinning frames
			$$\mathbf{p}\cdot R_T\colon \Real\to\SO(M)|_p.$$

		\begin{proposition}\label{substituteFrameSpinningPath}
			Let $\mathbf{p}=(\vec{t}_p, \vec{n}_p, \vec{t}_p\times\vec{n}_p)$ in $\SO(M)$
			be a frame at a point $p$ in $M$. Let $\phi\in\Real$ be a constant
			and denote by $\hat\omega$ the path of spinning frames 
			from $\mathbf{p}$ to $\mathbf{p}(\phi)$ which is parametrized by the interval $[0,\phi]$. 
			
			Given any positive constant $\delta$, and for every sufficiently large $L$ 
			with respect to $M$ and $\delta$, there exists a geodesic path $s$ in $M$ with both endpoints
			$p$ which satisfies the following requirements:
			\begin{itemize}
				\item The length of $s$ is $(\delta/L)$--close to $L$. The initial direction and the terminal direction
				of $s$ are both $(\delta/L)$--close 
				to $\vec{t}_p$.
				\item The parallel transport from $p$ back to $p$ along $s$ takes $\mathbf{p}$ to a frame
				$\mathbf{q}'$ which is $(\delta/L)$--close to $\mathbf{p}(\phi)$, and there exists a unique shortest
				path in $\SO(M)|_p$ between $\mathbf{q}'$ and $\mathbf{p}(\phi)$. 
				\item Denote by $\hat{s}$ the path which is the concatenation the parallel-transport
				path from $\mathbf{p}$ to $\mathbf{q}'$
				with the shortest path from $\mathbf{q}'$ to $\mathbf{p}(\phi)$.
				The relative homology class represented by $\hat{s}\in\pi_1(\SO(M),\mathbf{p},\mathbf{p}(\phi))$ equals 
				$[\hat{\omega}]\in H_1(\SO(M),\mathbf{p}\cup\mathbf{p}(\phi);\,\Integral)$.
			\end{itemize}					
		\end{proposition}
		
		\subsubsection{The local case} 
			Supposing that $\phi$ is given with $|\phi|$ small enough,
			we first show a local case of Proposition \ref{substituteFrameSpinningPath},
			namely, for any $\delta>0$ at most $|\phi|$, and any for sufficiently large $L>0$ 
			with respect to $M$ and $\delta$,
			the geodesic path $s$ in $M$ can be constructed with the asserted properties 
			of Proposition \ref{substituteFrameSpinningPath}.
						
			To be specific, it suffices to require $|\phi|$ to be so small that 
			the $(|\phi|\times10^4)$--neighborhood of $\mathbf{p}$ is embedded in $\SO(M)|_p$.
			Then the path of spinning frames $\mathbf{p}\cdot R_T\colon [0,10\phi]\to\SO(M)|_p$ 
			is the unique shortest path between $\mathbf{p}$ and $\mathbf{p}(10\phi)$
			in $\SO(M)|_p$. Then for any $0<\delta<|\phi|$, we require $L>1$ to be so large
			that the connection principle can be applied with respect to $\delta\times10^{-3}$.
			
			The construction of the local case is as follows.
			By the connection principle, we construct geodesic paths $a$ and $b$ 
			in $M$ with all endpoints $p$ of the following shape (or pose):
			\begin{itemize}
				\item The length of $a$ is $(\delta/100L)$--close to $L/4$. The same
				holds for the length of $b$.
				\item The initial direction and the terminal direction of $a$ are $(\delta/100L)$--close to
				$\vec{t}_p$ and $-\vec{t}_p$ respectively. Oppositely, 
				the initial direction and the terminal direction of $b$ are $(\delta/100L)$--close to
				$-\vec{t}_p$ and $\vec{t}_p$ respectively.
				\item The parallel transport along $a$ takes $\vec{n}_p$ back
				to $\vec{n}_p$ up to an error at most $\delta/100L$ in angle. 
				However, the parallel transport along $b$ takes $\vec{n}_p$ 
				to $\vec{n}_p(\phi/2)$ with error at most $\delta/100L$.				
			\end{itemize}
			The asserted path $s$ can be taken as the reduced concatenation of the consecutive chain
				$$a,b,\bar{a},\bar{b},$$
			where the bar notation stands for orientation reversal. In other words,
			$s$ is the unique geodesic path in $M$ which is homotopic to the concatenated path
			relative to the endpoints. 
			We check that the constructed path
			$s$ satisfies the requirements of Proposition \ref{substituteFrameSpinningPath}
			as follows.
			
			The first requirement about the length and directions at endpoints of $s$ follows
			from the length estimate of \cite[Lemma 4.8 (1)]{Liu-Markovic}.
			
			The second requirement about the parallel transport of $\mathbf{p}$ along $s$ 
			can be checked by considering the effect on the basis vectors. 
			The parallel transport of $\vec{t}_p$
			along $s$ ends up $(\delta/10L)$--close to $\vec{t}_p$ by the direction estimates
			of the first requirement with slightly more careful control of the error.
			The parallel transport of $\vec{n}_p$ consequentially along
			$a$, $b$, $\bar{a}$, $\bar{b}$ results in four vectors which are
			$(\delta/10L)$--close to $\vec{n}_p$, $\vec{n}_p(\phi/2)$, $\vec{n}_p(-\phi/2)$, $\vec{n}_p(\phi)$
			respectively. This is because the effect of parallel transport of vectors $\vec{u}\perp\vec{t}$
			along $a$ or $\bar{a}$ is approximately	the reflection about the axis $\Real\vec{n}_p$
			in the orthogonal complement of $\vec{t}_p$, while along $b$ or $\bar{b}$
			the axis is $\Real\vec{n}_p(\phi/4)$.
			Then the parallel transport of $\vec{n}_p$ along the reduced concatenation
			$s$ can be estimated by the phase estimate of \cite[Lemma 4.8 (2)]{Liu-Markovic}, 
			(choosing auxiliary framing of the segments $a$, $b$, $\bar{a}$, $\bar{b}$ 
			approximately the four vectors above). Combining the estimates for 
			the parallel transport of $\vec{t}_p$ and $\vec{n}_p$ yields the desired estimation
			for the distance between $\mathbf{q}'$ and $\mathbf{p}(\phi)$.			
			
			The third requirement on the homology class of $\hat{s}$
			is satisfied because of Lemmas \ref{estimateReduction} and \ref{estimateCommutator}.
			In fact, applying those lemmas with $\delta$ there taken to be $|10\phi|$, it follows that
			parallel transport along $s$ takes $\mathbf{p}$ to a frame $\mathbf{q}'\in\SO(M)|_p$
			which can be connected to $\mathbf{p}(\phi)$ along a unique shortest path $[\mathbf{q}',\mathbf{p}]$
			in $\SO(M)|_p$.
			Moreover, the parallel-transport path $[\mathbf{p},\mathbf{q}']$ concatenated with 
			the short path $[\mathbf{q}',\mathbf{p}]$ represents a $1$-cycle which is
			null homologous in $\SO(M)$. Since $\phi$ is small as chosen, the short path $[\mathbf{q}',\mathbf{p}]$ is 
			homotopic, relative to endpoints,
			to the concatenation of the shortest path $[\mathbf{q}',\mathbf{p}(\phi)]$ in $\SO(M)|_p$ 
			with the reversal of the frame-spinning path $\hat{\omega}=[\mathbf{p},\mathbf{p}(\phi)]$.
			
			Therefore, the path $s$ constructed above meets all the requirements of Proposition \ref{substituteFrameSpinningPath},
			as claimed for the local case.		
		
		\subsubsection{The general case}
			We proceed to prove the general case of Proposition \ref{substituteFrameSpinningPath}.
			Let $\mathbf{p}=(\vec{t}_p, \vec{n}_p, \vec{t}_p\times\vec{n}_p)$ in $\SO(M)$
			be a frame at a point $p$ in $M$, and let $\phi\in\Real$ be an arbitrarily given
			constant.
			
			Note that replacing $\phi$ with $\phi+4k\pi$ for any integer $k$ does not change the homology
			class of the frame-spinning path $[\hat\omega]\in H_1(\SO(M),\mathbf{p}\cup\mathbf{p}(\phi);\Integral)$,
			because $\pi_1(\SO(M)|_p)\cong\Integral_2$.
			Hence we may assume that $\phi\in[4\pi,8\pi)$, 
			and for some universally large integer $D>0$,
			say $10^5$, we may assume that the value
				$$\phi_D=\phi/D>0$$
			is as small as applicable to the local case. In particular,
			any ball of radius $1000\phi_D$ in $\SO(M)|_p$ is embedded.
			Given any $\delta>0$, we may assume that $1000\delta<4\pi/D$,
			possibly after replacing it with a smaller value.
			This ensures that any ball of radius $100\delta$ in $\SO(M)|_p$ is
			convex and isometrically embedded.
			Let $L>1$ be so large that $L_D=L/D$ 
			works for the local case with respect to $\delta_D=\delta/D^2$.
			
			Under the setting above,
			we construct the asserted path $s$ as follows.
			By applying the local case, we obtain a geodesic path $s_D$, 
			so that all the requirements of Proposition \ref{substituteFrameSpinningPath} 
			are satisfied by $s_D$ with respect to 
			the describing parameters $(\phi_D,\delta_D,L_D)$.
			We take the asserted geodesic path $s$ in $M$
			to be the reduced concatenation of the consecutive chain of paths in $M$:
				$$a_1,\cdots,a_D,$$
			which consists of $D$ copies $a_i$ of $s_D$.
			
			The first two requirements of Proposition \ref{substituteFrameSpinningPath} 
			can be shown to be satisfied by $s$ using a similar argument as the local case,
			which applies the estimates of \cite[Lemma 4.8]{Liu-Markovic}. 
			It remains to check that 
			the almost parallel-transport path
			$\hat{s}\in\pi_1(\SO(M),\mathbf{p},\mathbf{p}(\phi))$
			in the third requirement of Proposition \ref{substituteFrameSpinningPath}
			is homologous to 
			the frame-spinning path $\hat\omega$ in $H_1(\SO(M),\mathbf{p}\cup\mathbf{p}(\phi);\Integral)$.
			
			To this end, write $\hat\omega$ as the concatenation of consecutive subpaths $\hat\omega_1,\cdots,\hat\omega_D$,
			where $\hat\omega_i$ is the frame-spinning path parametrized by the subinterval $[(i-1)\phi_D,i\phi_D]$
			of $[0,\phi]$.
			The terminal endpoint of $\hat\omega_i$ is the frame $\mathbf{p}_i=\mathbf{p}(i\phi_D)$.
			Parallel transport of $\mathbf{p}$ consequentially
			along $a_1,\cdots,a_D$ gives rise to a sequence of points $\mathbf{q}_1,\cdots,\mathbf{q}_D$.
			We also set $\mathbf{q}_0=\mathbf{p}$.
			Since each $a_i$ is a copy of $s_D$, it is easy to estimate
			that parallel-transport along $a_i$ takes $\mathbf{p}_{i-1}$ to a frame $\mathbf{p}'_i$ in $\SO(M)|_p$,
			which	lies in the $(100\delta_D)$--neighborhood of $\mathbf{p}_i$;
			and hence, it can be estimated that $\mathbf{q}_i$ lies in the $(100\delta/D)$--neighborhood of $\mathbf{p}_i$.
			By our assumption, the balls $B_i$ of radius $100\delta$ centered at $\mathbf{p}_i$ 
			are mutually disjointly embedded in $\SO(M)|_p$, and any pair of frames in each $B_i$
			can be connected by a unique shortest path in $B_i$. The fact that $a_i$ is a copy of $s_D$
			implies that, as a $1$-chain of $\SO(M)$,
			the parallel-transport path $[\mathbf{q}_{i-1},\mathbf{q}_i]$
			equals $\hat\omega_i$ plus some $1$-chains in $B_{i-1}$ and $B_i$,
			modulo $1$-boundaries of $\SO(M)$.
			It follows that the consequential concatenation of $[\mathbf{q}_{i-1},\mathbf{q}_i]$, where $i=0,\cdots,D-1$,
			is equal to $\hat\omega$ plus $1$-chains in $B_i$ the modulo $1$-boundaries of $\SO(M)$.
			By Lemma \ref{estimateReduction}, the consequential concatenation of the $D$ paths 
			$[\mathbf{q}_{i-1},\mathbf{q}_i]$
			equals the parallel-transport path $[\mathbf{p},\mathbf{q}']$ of $\mathbf{p}$ along $s$,
			up to $1$-chains of $B_D$ and $1$-boundaries of $\SO(M)$.
			It follows that the almost parallel-transport path
			$\hat{s}$ in the third requirement of Proposition \ref{substituteFrameSpinningPath}
			differs from $\hat\omega$ only by $1$-boundaries of $\SO(M)$ and $1$-cycles of $B_i$, which are again
			$1$-boundaries of $\SO(M)$. 
			This verifies the third requirement of Proposition \ref{substituteFrameSpinningPath}.
			
			Therefore, we have completed the proof of Proposition \ref{substituteFrameSpinningPath}
			in the general case.			
		
	\subsection{Connecting frames in a grading class}
		In this section, we prove Theorem \ref{gradedConnectionPrinciple}.
		Suppose that
		$\mathbf{p}=(\vec{t}_p,\vec{n}_p,\vec{t}_p\times\vec{n}_p)$ and 
		$\mathbf{q}=(\vec{t}_q,\vec{n}_q,\vec{t}_q\times\vec{n}_q)$
		in $\SO(M)$ be frames given	at points $p$ and $q$ of $M$, respectively.
		For any given relative homology class $\Xi\in H_1(\SO(M),\mathbf{p}\cup\mathbf{q};\,\Integral)$,
		the goal is to construct a geodesic segment $s$ in $M$ from $p$ to $q$
		so that the parallel-transportation path of $\mathbf{p}$ along $s$ almost
		represents $\Xi$. Furthermore, for any given constant $\delta>0$,
		we want $s$ to have length approximately $L$, 
		and transport $\mathbf{p}$ approximately to $\mathbf{q}$,
		with error bounded by $\delta/L$,
		and we want the construction to be possible for any sufficiently large $L>0$.
		
		The recipe consists of three steps: First, we construct a geodesic segment $w$
		from $p$ to $q$ which realizes the projection of $\Xi$ in $H_1(M,p\cup q;\Integral)$.
		Then we adjust $w$ by concatenating it nearly smoothly with null homologous 
		closed geodesic paths $a,b$ based at $p,q$ respectively, so that
		the reduced concatenation $awb$ have approximately the wanted directions
		at the endpoints. Finally, we adjust the $awb$ by further concatenating it with
		a frame-spinning path $z$ based at $q$, in a nearly smooth fashion, 
		so that the parallel-transport path induced by the reduced concatenation $awbz$ 
		almost represents $\Xi$.
		We can appropriately choose the length of the paths $a,w,b,z$ so that their reduced concatenation
		$awbz$ yields the desired segment $s$.
		
		To be precise, suppose that a constant $\delta>0$ is given. Possibly after replacing it
		with a smaller value (depending the injectivity radius of $M$),
		we may assume in addition that any ball of radius $1000\delta$ is 
		embedded in $\SO(M)$. 
		At this point, we simply assume that $L>1$ is 
		any arbitrary constant
		which is sufficiently large 
		to enable all the participating constructions and estimates
		in our recipe.
		A specific lower bound for $L$ is summarized at the end of our recipe.
		
		The recipe for constructing the asserted $s$ is as follows.
		
		\emph{Step 1}.		
		We take $w$ to be a geodesic path in $M$ from $p$ to $q$ which realizes
		the projection of $\Xi$ in $H_1(M,p\cup q;\Integral)$. Here the projection is induced by 
		the projection map $\SO(M)\to M$. Since all but finitely many such $w$ are long, 
		for any constant $K_0=K_0(\delta)>1$, to be specified later, we may assume that
		the length of $w$ is at least $K_0$. 
		On the other hand, there exists a constant 
			$$K_1=K_1(K_0,M,\Xi)>K_0,$$
		so that the length of $w$ is at most $K_1$,
		such as the minimal length of paths in the projection of $\Xi$
		of length at least $K_0+1$. Denote the length of $w$ by
			$$K_0<\ell(w)<K_1.$$
		
		\emph{Step 2}.	
		We take $a$ to be a geodesic path in $M$ from $p$ to $p$, and $b$
		a geodesic path in $M$ from $q$ to $q$, so that the following properties hold.
		\begin{itemize}
			\item The length $\ell(a)$ of $a$ is $(\delta/100L)$--close to $L/4-\ell(w)/2$.
			The initial direction of $a$ is $(\delta/100L)$--close to $\vec{t}_p$, and the terminal direction of
			$a$ is $(\delta/100L)$--close to the initial direction of $w$. Moreover, $a$ is null homologous in $M$.
			\item The length $\ell(b)$ of $b$ is $(\delta/100L)$--close to $L/4-\ell(w)/2$.
			The initial direction of $b$ is $(\delta/100L)$--close to the terminal direction of $w$, and the terminal direction of
			$b$ is $(\delta/100L)$--close to $\vec{t}_q$. Moreover, $b$ is null homologous in $M$.
		\end{itemize}
		The geodesic path $a$ can be taken to be the reduced concatenation
		of a commutator of two suitably chosen geodesic paths at $p$,
		and similarly can we construct $b$. 
		An explicit construction can be found in \cite[Lemma 3.4]{Sun-virtualDomination}. For example, 
		to invoke the lemma to
		create a geodesic path $a$, one may take the vectors $\vec{v}_1$ and $\vec{v}_2$ there
		to be $\vec{t}_p$ and the initial direction of $w$, and the vector $\vec{n}$ 
		there to be any unit vector	perpendicular to both $\vec{v}_1$ and $\vec{v}_2$.
		Let 
		$$L_1=L_1(\delta,K_1,M)>10K_1$$
		be an \textit{a priori} lower bound for $L$ to enable the construction
		for $a$ and $b$. Note that $\ell(a)$ and $\ell(b)$ are hence at least $K_0$.
				
		\emph{Step 3}. We take $z$ to be a geodesic path in $M$ from $q$ to $q$ as follows.
		Recall that the notation $\mathbf{q}(\phi)$ stands for the rotation of $\mathbf{q}$
		about $\vec{t}_q$ by an angle $\phi$. By our construction,
		the parallel transport of $\vec{n}_p$ consequentially along $a,w,b$ is
		a unit vector $\vec{m}$ at $q$ which is nearly orthogonal to $\vec{t}_q$.
		Pick a value $\phi\in\Real$ such that $\vec{n}_q(-\phi)$ minimizes the distance
		to $\vec{m}$ among unit vectors orthogonal to $\vec{t}_q$.
		We apply Proposition \ref{substituteFrameSpinningPath} to construct 
		two candidates of $z$, denoted as $z_\uparrow$ and $z_\downarrow$.
		One of them satisfies the following properties, and the other
		satisfies the same properties but with $\phi$ replaced by $\phi+2\pi$.
		\begin{itemize}
				\item The length of $z$ is $(\delta/100L)$--close to $L/2$. The initial direction and the terminal direction
				of $z$ are both $(\delta/100L)$--close 
				to $\vec{t}_q$.
				\item The parallel transport from $q$ back to $q$ along $z$ takes $\mathbf{q}(-\phi)$ to a frame
				$\mathbf{q}'$ which is $(\delta/100L)$--close to $\mathbf{q}$, and there exists a unique shortest
				path in $\SO(M)|_p$ between $\mathbf{q}'$ and $\mathbf{q}$. 
				\item Denote by $\hat{z}$ the path which is the concatenation the parallel-transport
				path from $\mathbf{q}(-\phi)$ to $\mathbf{q}'$
				with the shortest path from $\mathbf{q}'$ to $\mathbf{q}$.
				The relative homology class represented by $\hat{z}\in\pi_1(\SO(M),\mathbf{q}(-\phi),\mathbf{q})$ equals 
				$[\hat{\omega}]\in H_1(\SO(M),\mathbf{q}(-\phi)\cup\mathbf{q};\,\Integral)$.
		\end{itemize}					
		Note that $\mathbf{q}(-\phi)=\mathbf{q}(-\phi-2\pi)$ but the corresponding $[\hat{\omega}]$
		are different for the two candidates, so the relative homology classes 
			$$[\hat{z}_\uparrow],\,[\hat{z}_\downarrow]\,\in\, H_1(\SO(M),\mathbf{q}(-\phi)\cup\mathbf{q};\,\Integral)$$
		differ exactly by the canonical element $[\hat{c}]\in H_1(\SO(M);\Integral)$ of order $2$, 
		namely, the nontrivial element of $H_1(\SO(M)|_q;\Integral)\cong\Integral_2$.
		Let
			$$L_2=L_2(\delta,K_0,M)>K_0$$
		be a lower bound for $L$ so that the application of Proposition \ref{substituteFrameSpinningPath} is valid.
		
		The geodesic path $w,a,b,z_\uparrow,z_\downarrow$ that we have constructed
		are all longer than the presumed constant $K_0>1$. 
		According to \cite[Lemma 4.9]{Liu-Markovic} and Lemma \ref{estimateReduction}
		we can choose 
			$K_0=K_0(\delta)>1$ 
		to be sufficiently large,
		then both of the geodesic paths $s_\uparrow=awbz_\uparrow$ and $s_\downarrow=awbz_\downarrow$,
		which are obtained by reduced concatenation,
		satisfy the first two asserted properties of Theorem \ref{gradedConnectionPrinciple}.
		Regarding to the third property,
		$s_\uparrow$ and $s_\downarrow$ give rise to two paths of frames $\hat{s}_\uparrow$ and $\hat{s}_\downarrow$
		from $\mathbf{p}$ to $\mathbf{q}$, and exactly one of them represents $\Xi\in H_1(\SO(M),\mathbf{p}\cup\mathbf{q};\Integral)$
		while the other represents $\Xi+[\hat{c}]$. 
		
		To finish our recipe, we take the correct $z$ and the corresponding reduced concatenation
			$$s=awbz$$
		so that $\hat{s}$	represents $\Xi$. 
		Then the geodesic path $s$ in $M$ from $p$ to $q$ meets all the asserted properties
		of Theorem \ref{gradedConnectionPrinciple}.
		
		To summarize, suppose that $\mathbf{p}$, $\mathbf{q}$ in $\SO(M)$ are given.
		For any constant $\delta>0$, we choose constants $K_0$, $K_1$, $L_1$, and $L_2$ in order,
		and take
			$$L_3=L_3(\delta,M,\Xi)\,=\max(L_1,L_2).$$
		Then for any constant $L>L_3$, the asserted geodesic path $s$ in $M$ from $p$ to $q$
		of Theorem \ref{gradedConnectionPrinciple} can be construction following the recipe above.
		
		This completes the proof of Theorem \ref{gradedConnectionPrinciple}.
			
	\subsection{Proof of Theorem \ref{constructGradedGoodCurves}}
		We derive Theorem \ref{constructGradedGoodCurves} 
		from the relative version Theorem \ref{gradedConnectionPrinciple}
		as follows.
		
		Given any homology class $\Xi\in H_1(\SO(M);\Integral)$ and any constant $\epsilon>0$, 
		we take an arbitrary frame $\mathbf{p}=(\vec{t}_p,\vec{n}_p,\vec{t}_p\times\vec{n}_p)
		\in\SO(M)|_p$ at a point $p\in M$.
		Identify $H_1(\SO(M);\Integral)$ naturally with $H_1(\SO(M),\mathbf{p};\Integral)$.
		Denote by 
		$$\hat{c}\,\in\,\pi_1(\SO(M))$$	
		the nontrivial central element, which is unique of order $2$.
		For sufficiently large $R$ with respect to $\epsilon$, $M$, and $\Xi+[\hat{c}]$,
		we apply Theorem \ref{gradedConnectionPrinciple} to construct
		a geodesic path $s$ in $M$ from $p$ to $p$ with the following properties:
		\begin{itemize}
		\item The length of $s$ is $(\epsilon/10R)$--close to $R$. The initial direction and 
		the terminal directions are both $(\epsilon/10R)$--close to $\vec{t}_p$.
		\item The parallel transport of $\mathbf{p}$ along $s$ is a frame $\mathbf{p}'\in\SO(M)|_p$ which is
		$(\epsilon/10R)$--close to $\mathbf{p}$.
		\item The asserted concatenated path $\hat{s}$ represents 
		$\Xi+[\hat{c}]$ in $H_1(\SO(M),\mathbf{p};\Integral)$.
		\end{itemize}
		Denote by $\gamma$ the unique free geodesic loop which is freely homotopic to the closed path $s$.
		
		It follows from Lemma \ref{estimateCyclicReduction} that $\gamma$ is an $(R,\epsilon/R)$--curve. 
		Moreover, for any frame $\mathbf{q}$ at a point $q\in\gamma$, 
		the cyclic concatenation  of the parallel-transport path 
		$[\mathbf{q},\mathbf{q}']$ of $\mathbf{q}$ around $\gamma$
		with the shortest path $[\mathbf{q}',\mathbf{q}]$ in $\SO(M)|_q$
		is freely homotopic to the closed path $\hat{s}$ in $\SO(M)$, 
		but its free homotopy class
		differs from the canonical lift $\hat{\gamma}\in\pi_1(\SO(M))$
		by a factor $\hat{c}$,
		(see Section \ref{Sec-preliminaries}).
		Therefore, $[\hat\gamma]$ equals $\Xi$ in $H_1(\SO(M);\Integral)$, as desired.
		
		This completes the proof of Theorem \ref{constructGradedGoodCurves}.

\section{Subsurface of odd Euler characteristic}\label{Sec-oddEulerCharacteristic}
	In this section, we construct immersed quasi-Fuchsian subsurfaces
	of odd Euler characteristic in closed hyperbolic $3$-manifolds, 
	proving Theorem \ref{main-oddEulerCharacteristic}.
	
	Let $M$ be a closed hyperbolic $3$-manifold. 
	Without loss of generality,
	we assume that $M$ is orientable and fix an orientation, 
	otherwise we pass to an orientable double cover.
	Denote by $\SO(M)$ the special orthonormal frame bundle over $M$.
	
	Take an arbitrary frame in $\SO(M)$ at a point $p$ in $M$, denote as
		$$\mathbf{p}=(\vec{t}_p,\vec{n}_p,\vec{t}_p\times\vec{n}_p).$$	
	Denote by $\mathbf{p}^*\in\SO(M)|_p$ the frame
		$$\mathbf{p}^*=(\vec{t}_p,-\vec{n}_p,-\vec{t}_p\times\vec{n}_p).$$
	For a sufficently small constant $\epsilon>0$ and some sufficiently large constant $R>1$,
	which we specify in the summary at the end of the proof,
	we apply Theorem \ref{gradedConnectionPrinciple} 
	to construct a geodesic path $s$ in $M$ with both endpoints $p$
	satisfying the following requirements:
	\begin{itemize}
	\item The length of $s$ is $(\epsilon/100R)$--close to $R/2$.
	The initial direction and the terminal direction of $s$ are both $(\epsilon/100R)$--close to $\vec{t}_p$.
	\item The parallel transport of $\mathbf{p}$ along $s$ back to $p$ is $(\epsilon/100R)$--close to $\mathbf{p}^*$.
	\item The closed path $s$ in null homologous in $H_1(M;\Integral)$.
	\end{itemize}
	Note that there are two qualified candidates 
	for the relative homology class $\Xi=[\hat{s}]$
	in $H_1(\SO(M),\mathbf{p}\cup\mathbf{p}^*;\,\Integral)$ as of Theorem \ref{gradedConnectionPrinciple},
	differing from each other 
	by the homology class of the central nontrivial element $\hat{c}\in\pi_1(\SO(M))$.
	Either of them works fine.
	For sufficently small $\epsilon$ and large $R$, the closed geodesic loop of $M$ 
	which is freely homotopic to the cyclic concatenation of two copies of $s$ is a good curve:
		$$\gamma\in\ocurves_{R,\epsilon}(M).$$
	The closed geodesic loop of $M$ freely homotopic $s$ itself, denoted as
		$$\sqrt{\gamma}\in\pi_1(M),$$
	is not good. It is doubly covered by $\gamma$, and has complex length 
	approximately $R/2+\pi\cdot\sqrt{-1}$, with error at most $\epsilon/2R$
	in absolute value.
	
	For sufficiently small $\epsilon$ and large $R$, we observe the following fact:
	
	\begin{lemma}\label{gammaBounds}
				The canonical lift $\hat\gamma\in\pi_1(\SO(M))$ is null homologous in $\SO(M)$.
	\end{lemma}
	
	\begin{proof}
		Since the element $s\in\pi_1(M,p)$ is homologically trivial,
		there exists a path of frames $\hat\alpha\in\pi_1(\SO(M),\mathbf{p},\mathbf{p}^*)$ 
		which is contained in $\SO(M)|_p$, such that the path $\hat{s}$ from $\mathbf{p}$
		to $\mathbf{p}^*$ as in the conclusion of Theorem \ref{gradedConnectionPrinciple}
		is relatively homologous to $\hat\alpha$. Denote by $R_T(\pi)\in\SO(3)$ the matrix
		$$R_T(\pi)\,=\,\left[
			\begin{array}{ccc} 
			1&0&0\\
			0&-1&0\\
			0&0&-1
			\end{array}
			\right].$$
		It follows that the path
			$$\alpha^*=\alpha\cdot R_T(\pi)$$
		from $\mathbf{p}^*$ to $\mathbf{p}$ is relatively homologous to
		$\hat{s}\cdot R_T(\pi)$.
		By Lemma \ref{estimateCyclicReduction} and the definition of the canonical lift, 
		the cyclic concatenation of $\hat{s}$ with $\hat{s}\cdot R_T(\pi)$ differs from $\hat\gamma$
		by a factor $\hat{c}$, as an element of $\pi_1(\SO(M))$.
		On the other hand, the concatenation of $\alpha$ with $\hat\alpha^*$ represents
		$\hat{c}$ in $\pi_1(\SO(M))$ by the topology of $\SO(3)$.
		We conclude that $\hat\gamma$ is homologically $2[\hat{c}]=0$ in $H_1(\SO(M);\Integral)$.				
	\end{proof}
	
	Therefore, it follows from Theorem \ref{solutionRelative} that there exists
	a $\pi_1$-injectively immersed, $(R,\epsilon)$--panted, connected, quasi-Fuchsian subsurface
		$$F\looparrowright M$$
	which is oriented and bounded by $\gamma$.
	In fact, the complex Fenchel--Nielsen coordinates associated to any glued cuff $C$ of $F$ 
	can be required to satisfy $(\chlen(C),s(C))\approx (R/2,1)$ with error
	at most $(\epsilon/R,\epsilon/R^2)$ componentwise,
	the same as in the original construction of Kahn and Markovic 
	(Theorem \ref{surfaceSubgroupTheorem} and Declaration \ref{finerError}).

	We would like to identify every pair of antipodal points of $\gamma$,
	or topologically to glue to the boundary of $F$ a M\"obius band
	with the core $\sqrt{\gamma}$, so as to produce an immersed non-orientable
	subsurface $\Sigma$ as our output. However, to guarantee that $\Sigma$ is $\pi_1$--injective,
	we want to make sure that geometrically $F$ contains no properly embedded essential arcs
	which are relatively short compared to $R$.
	This is the last technical point that we need to address 
	before completing the construction.
	
	To this end, denote by $\mathcal{G}(F)$ the dual graph of the inherited pants decomposition of $F$.
	Namely, the vertices of $\mathcal{G}(F)$ are the $(R,\epsilon)$--pants of $F$, and the edges
	of $\mathcal{G}(F)$ are the glued cuffs,
	which are the $(R,\epsilon)$--curves in the interior of $F$.
	The graph $\mathcal{G}(F)$ is trivalent except at
	the distinguished valence-$2$ vertex $P_0$ which contains $\partial F$.
	
	\begin{lemma}\label{longCombinatorialPath}
		The panted subsurface $F$ can be constructed to satisfy 
		the extra condition	that every non-contractible 
		closed combinatorial path of $\mathcal{G}(F)$
		based at $P_0$ has combinatorial length at least $Re^{R/4}$.
	\end{lemma}
	
	\begin{proof}
		Suppose that $F_0$ is an oriented connected $(R,\epsilon)$--panted subsurface
		with connected boundary $\gamma$, as guaranteed by Theorem \ref{solutionRelative}.
		We may require $F_0$ to have complex Fenchel--Nielsen coordinates
		$(\epsilon/R,\epsilon/R^2)$--close to $(R/2,1)$ for every glued cuff.
		See \cite[Corollary 2.7]{Sun-virtualDomination} for an outline of the construction
		based on \cite{Liu-Markovic}.
		
		Denote by $P_0$ the distinguished pair of pants of $F_0$ with one cuff $\partial F$.
		Denote by $C_1,C_2$ the other two cuffs of $P_0$, and $P_1,P_2$ the other two pairs of 
		pants adjacent to $P_0$ along $C_1,C_2$ accordingly.
		Take an oriented connected closed $(R,\epsilon)$--panted surface $E$ with complex Fenchel--Nielsen
		coordinates $(\epsilon/R,\epsilon/R^2)$--close to $(R/2,1)$, (Theorem \ref{surfaceSubgroupTheorem}).
		We may require that (the pants types of) $P_1$ and $P_2$ 
		also appear in the inherited pants decomposition of $E$.
		This follows from \cite[Theorems 2.9 and 2.10]{Liu-Markovic}.
		Possibly after passing to a finite cover of $E$ induced by a regular finite cover of the dual graph $\mathcal{G}(E)$,
		we may assume that every embedded cycle of $\mathcal{G}(E)$ has combinatorial length at least $Re^{R/4}$
		and no edge of $\mathcal{G}(E)$ is separating.
		
		We modify $F_0$ to obtain a new subsurface $F$ with the asserted property.
		Take a copy $E'$ of $E$ such that 
		some pair of pants $P'_1\subset E'$ has the same pants type of $P_1$, 
		with a cuff $C'_1$ corresponding to $C_1$. Denote by $P'\subset E'$ the pants adjacent to $P'_1$
		along $C'_1$. We make a cross change between $F_0$ and $E'$ along the parallel glued cuffs
		$C_1$ and $C'_1$. Namely, cut $F_0$ along $C_1$, and $E'$ along $C'_1$; 
		identify the new unglued cuff of $P'_1$ 
		with the new unglued cuff of $P_0$, 
		and similarly glue the pants $P_1$ and $P'$ along their new unglued cuffs.
		In the same way, we make a cross change between $F_0$ and another copy $E''$ of $E$
		along cuffs corresponding to $C_2$. In effect, we obtain a new $(R,\epsilon)$--panted 
		subsurface $F$ which is oriented and bounded by $\gamma$. The complex Fenchel--Nielsen coordinates
		remain $(\epsilon/R,\epsilon/R^2)$--close to $(R,1)$.
		The pants decomposition graph $\mathcal{G}(F)$ is obtained from $\mathcal{G}(F_0)$ and two copies of $\mathcal{G}(E)$
		by two cross changes of edges corresponding to $C_1$ and $C_2$.
		It is straightforward to see that $\mathcal{G}(F)$ has no non-contractible closed path
		based at $P_0$ of combinatorial length smaller than $Re^{R/4}$.
		Therefore, the panted subsurface $F$ is as desired.	
	\end{proof}
	
	Take $F$ to be a connected oriented $(R,\epsilon)$--panted immersed subsurface of $M$ bounded by 
	the $(R,\epsilon)$--curve $\gamma$,
	such that the complex Fenchel--Nielsen coordinates 
	are approximately $(R/2,1)$ for every glued cuff, with error at most $(\epsilon/R,\epsilon/R^2)$ componentwise.
	Suppose in addition that the pants decomposition graph $\mathcal{G}(F)$
	satisfies the conclusion of Lemma \ref{longCombinatorialPath}.
	Since $\gamma$ doubly cover the closed geodesic $\sqrt{\gamma}$ of $M$, 
	the pre-image of any point of $\sqrt\gamma$ is a pair of antipodal points in $\gamma$.
	We identify every pair of antipodal points of $\partial F$ accordingly.
	
	The result is a connected closed immersed subsurface
		$$\Sigma\looparrowright M,$$
	which is no longer orientable. However, it is $\pi_1$--injectively immersed
	and geometrically finite, hence quasi--Fuchsian, 
	by a geometric criterion due to H.~Sun \cite[Theorem 2.6]{Sun-virtualTorsion}.
	Furthermore, the Euler characteristic of $S$ can be computed by:
		$$\chi(\Sigma)\,=\,\chi(F)\,=1-2\cdot\mathrm{genus}(F),$$
	which is an odd number.
	
	In summary, we can choose some sufficiently small $\epsilon>0$ according to $M$
	and some sufficiently large $R>0$ according to $M$ and $\epsilon$, and construct
	a closed quasi--Fuchsian subsurface $\Sigma\looparrowright M$ of odd Euler characteristic as asserted.
	Fixing a choice of $\epsilon$, it suffices to require that $R$ should be so large  
	that all the constructions and estimates that we have done work.
	
	This completes the proof of Theorem \ref{main-oddEulerCharacteristic}.

\section{Irregular exhausting tower with exponential torsion growth}\label{Sec-irregularExhaustingTower}
	
	In this section, 
	we present the core construction of Theorem \ref{exponentialTorsionGrowth},
	which can be stated in slightly more details as follows.
	
	\begin{proposition}\label{exponentialTorsionGrowthSpecial}
		Let $M$ be an orientable closed hyperbolic $3$-manifold.
		Suppose that there are closed surfaces $S$ and $\Sigma$,
		and there are maps $\iota_S\colon S\to M$, and $\iota_\Sigma\colon \Sigma\to M$,
		and $\rho\colon M\to S$ with the following properties:
		\begin{itemize}
			\item The surface $S$ is orientable, and $\Sigma$ of odd Euler characteristic.
			\item The maps $\iota_S$ and $\iota_\Sigma$ are embeddings with mutually disjoint images.
			\item The composition $\rho\circ\iota_S$ is homotopic to the identity, and
			the composition $\rho\circ\iota_\Sigma$ is homotopic to a constant point map.
		\end{itemize}
		Then for any constant $0<\epsilon<1$ and any base point $*$ of $M$, 
		there exists a tower of finite covers of $M$ with distinguished lifted base points:
		$$\cdots\longrightarrow (\tilde{M}_n,\tilde{o}_n)\longrightarrow\cdots\longrightarrow (\tilde{M}_2,\tilde{o}_2)\longrightarrow 
		(\tilde{M}_1,\tilde{o}_1)\longrightarrow (M,*).$$
		Moreover, the following requirements are satisfied:
		\begin{itemize}
			\item The injectivity radii of $\tilde{M}_n$ at the base points are unbounded.
			\item The number of distinct lifts of $\iota_\Sigma$ into each $\tilde{M}_n$ is at least $(1-\epsilon)\cdot[\tilde{M}_n:M]$.
			Hence, 
			$$\lim_{n\to\infty}\frac{\log|\mathrm{Tor}_1(H_1(\tilde{M}_n;\Integral),\Integral_2)|}{[\tilde{M}_n:M]}\geq (1-\epsilon)\log2.$$
		\end{itemize}
	\end{proposition}
	
	To illustrate the strategy, let us briefly explain the first move.
	That is, given a nontrivial element $g_1\in\pi(M,*)$, 
	how to construct a finite cover $(\tilde{M}_1,\tilde{o}_1)$ into which $g_1$ does not lift but many $\Sigma$ lifts.
	We can first take a cyclic cover $(M'_1,o'_1)$ dual to $S$ of some large degree $d$.
	Writing $V$ for the complement of $S$ in $M$, we can decompose $M'_1$ into $d$ lifted copies 
	$V_i$ of $V$, where $i\in\Integral_d$.
	If $g_1$ intersects $S$ algebraically nontrivially, we can simply take $d$ to be very large
	and $\tilde{M}_1$ to be $M'_1$.
	Otherwise, $g_1$ lifts into $M'_1$ based at $o'_1$, and it is contained
	in the union $W$ of some consecutive pieces $V_i$ nearby. 
	As the number of pieces in $W$ depends only on $g_1$, 
	$W$ occupies a very small portion of $M'_1$ if we choose $d$ to be large.
	To construct $\tilde{M}_1$, the idea is to assemble a finite cover of $W$ and 
	some finite covers of other pieces
	$V_j$. We want to require that $g_1$ does not lift to the finite cover of $W$, and meanwhile,
	that every preimage component of $\Sigma$ in the finite cover of any other piece $V_j$ is a lift of $\Sigma$.
	The latter can be
	ensured if the finite cover of $V_j$ comes from a finite cover of $S$, via the inclusion of $V$ into $M$
	and the retraction of $M$ onto $S$. The former can be ensured by the residual finiteness of $W$.
	However, in order to glue these individual covering pieces together, 
	it is crucial to know that their restriction to the boundary are isomorphic	covers of $S$,
	better characteristic. This is guaranteed by the so-called omnipotence lemma (Lemma \ref{omnipotence}),
	a consequence of Wise's Malnormal Special Quotient Theorem.
	Assume that all these are done, then we will obtain a desired $\tilde{M}_1$
	in which $g_1$ disappears but most lifts of $\Sigma$ survive.
	
	The construction can be iterated to 
	destroy more elements from $\pi_1(M,*)$, one at each time.
	If the degrees of cyclic covers are chosen to grow very fast,
	the total portion of non-lifts (or `damaged lifts') of $\Sigma$ will remain small.
	To argue by induction in a formal way, 
	we introduce some notion called \emph{generalized digital expansion} to encode a cyclic covering tower
	$\{M'_n\}$ underlying the asserted $\{\tilde{M}_n\}$.
	A precise induction hypothesis can be found as the statement of Lemma \ref{constructionAssertedTower}.

	The rest of this section is devoted to the proof of Proposition \ref{exponentialTorsionGrowthSpecial}.
	
	\subsection{Cyclic towers encoded by generalized digital expansions}
	In this subsection, we introduce
	a tower of finite cyclic covers which is to be considered
	as an intermediate step toward 
	the construction of the exhausting tower asserted by Proposition \ref{exponentialTorsionGrowthSpecial}.
	
	For this subsection, we suppose that 
	$(M,*)$ is an orientable closed $3$-manifold,
	and $S$ is an embedded non-separating oriented connected closed subsurface of $M$
	which misses $*$.
	Denote by
		$$V\,=\,M\setminus\mathrm{Nhd}^\circ(S)$$
	the compact submanifold of $M$ obtained by cutting along $S$.
	Let $d$ be an odd positive integer which is at least $3$.

	\subsubsection{Generalized digital expansion and blocks}
	For any positive integer $n$, denote by $[d^n]$ the set of 
	the $d^n$ consecutive integers centered at $0$, namely,
		$$[d^n]\,=\,\left\{0,\pm1,\pm2,\cdots,\pm\frac{d^n-1}2\right\}.$$
	Adopt the notation 
		$$s_n=1+2+\cdots+n=\frac{n(n+1)}2.$$
	There is a canonical bijective correspondence between sets:
	\begin{eqnarray*}
		[d^{s_n}]&\longleftrightarrow&[d]\times[d^2]\times\cdots\times [d^n]\\
		a&\leftrightarrow&(a_0,\,a_1,\,\cdots,\,a_{n-1})
	\end{eqnarray*}
	which is determined by the relation
		$$a\,=\,\sum_{i=0}^{n-1} a_j d^{s_j}.$$
	We view the correspondence as a generalized digital expansion
	for the integers $a$ of $[d^{s_n}]$:
	At the $j$-th place with the assigned weight $d^{s_j}$, the digit $a_j$
	is taken from the digit set $[d^{j+1}]$, which is particular for that place.
	
	We say that two integers $a,a'\in[d^{s_n}]$
	are in \emph{the same block}, denoted as
		$$a\sim a',$$
	if the generalized digital expansions of $a$ and $a'$
	agree from the $(n-1)$--th place all the way down to 
	the highest place with a digit $0$.
	In other words,
	if $a_j\neq0$ for all $j=0,\cdots,n-1$, 
	then $a\sim a'$ means $a_j=a'_j$ for all $j=0,\cdots,n-1$;
	if $a_k=0$ for some $k\in\{0,\cdots,n-1\}$
	and $a_j\neq0$ for all $j=k+1,\cdots,n-1$, 
	then $a\sim a'$ means $a_j=a'_j$ for all $j=k,\cdots,n-1$.
	Since being in the same block is an equivalence relation on $[d^{s_n}]$,
	we call the equivalence classes the \emph{blocks} of $[d^{s_n}]$.
	Denote the set of blocks of $[d^{s_n}]$ as
		$$\mathscr{B}_n=\mathscr{B}_n(d)=[d^{s_n}]/\sim.$$
	For any block $\beta\in\mathscr{B}_n$, the \emph{level} of $\beta$
	is said to be $k\in\{0,\cdots,n-1\}$,
	if for some (hence any) integer $a\in\beta$, 
	$a_k$ is the highest $0$ in the expansion of $a$; 
	the level of $\beta$ is formally defined to be $\infty$, 
	if the block $\beta$ consists of a single integer which has no $0$ digits.
	
	For example, the subset $[d^{s_{n-1}}]$ of $[d^{s_n}]$ 
	is the only block of level $(n-1)$. 
	Its size is $1/d^n$ of $[d^{s_n}]$.
	Any of the remaining blocks of $[d^{s_n}]$, say of a level $j$
	other than $\infty$, looks like a shifted subset $[d^{s_j}]$
	centered at some integer with an expansion $(0,\cdots,0,a_{j+1},\cdots,a_{n-1})$,
	where $a_{j+1},\cdots,a_{n-1}$ are nonzero.
	For large $n$, the blocks of level $\infty$ in $[d^{s_n}]$ 
	reminds us of the picture of Cantor's dust set, 
	but the dust is really heavy:
		
	\begin{lemma}\label{heavyDust}
		Given any constant $0<\epsilon<1$, the following statement
		holds true for all sufficiently large odd positive integers $d$:
		For all $n\in\Natural$, the number of blocks of level $\infty$ in $[d^{s_n}]$
		is greater than $(1-\epsilon)\cdot d^{s_n}$.	
	\end{lemma}
	
	\begin{proof}
		The number $C_n$ of the blocks of level $\infty$ in $[d^{s_n}]$
		is clearly	$(d-1)\times\cdots\times(d^n-1)$.
		For an odd positive integer $d$ at least $3$, 
		the portion of such blocks $(C_n\,/\,d^{s_n})$ is strictly decreasing as $n$ grows.
		The limit can be expressed as $\phi(1/d)$ using the Euler function:
			\begin{eqnarray*}
			\phi(q)&=&\prod_{j=1}^\infty(1-q^n)\,=\,\sum_{k=-\infty}^{\infty}(-1)^kq^{(3k^2-k)/2}\\
			&=&1-q^2-q^3+q^5+q^7-\cdots
			\end{eqnarray*}
		When $d$ is sufficiently large, we have $\phi(1/d)>1-\epsilon$,
		so $C_n$ is at least $(1-\epsilon)\cdot d^{s_n}$.
	\end{proof}
	
	\subsubsection{Encoding a tower of finite cyclic covers} 
	Continue to adopt the notations of the generalized digital expansion
	with respect to $d$. Denote by
		$$\cdots\longrightarrow M'_n\longrightarrow \cdots\longrightarrow M'_2\longrightarrow M'_1\longrightarrow M$$
	the tower of finite cyclic covers dual to $S$ with the covering degree
		$$[M'_n:M]\,=\,d^{s_n},$$
	so $M'_n$ covers $M'_{n-1}$ cyclically of degree $d^n$.		
	
	Since $V$ is the compact submanifold obtained by cutting $M$ along $S$,
	the boundary $\partial V$ has two components $\partial_\pm V$ parallel to $S$ inside $M$
	such that the (outward) induced orientation of $\partial_\pm V$ coincides with $\mp S$.
	The cyclic cover $M'_n$ can be constructed by gluing $d^{s_n}$ copies $V_n(a)$ of $V$ in such a way
	that $\partial_+ V_n(a)$ is identified with $\partial_- V_n(a+1)$ for all $a\in[d^{s_n}]$.
	By convention, $V_n(\frac{d^{s_n}-1}2+1)$ stands for $V_n(-\frac{d^{s_n}-1}2)$.
	Moreover, each block $\beta\in\mathscr{B}_n$ corresponds to a \emph{block of pieces} $W_n(\beta)$ of $M'_n$
	obtained by gluing the \emph{unit pieces} $V_n(a)$ for all $a\in\beta$.
	
	Therefore, we have the following decompositions of $M'_n$ by various lifts of $S$:
		$$M'_n\,=\,\bigcup_{a\in[d^{s_n}]}\,V_n(a)\,=\,\bigcup_{\beta\in\mathscr{B}_n}\,W_n(\beta).$$
	Equip each $M'_n$ with the lifted base point 
	$o'_n$
	which lies in the unit piece $V_n(0)$.
	Therefore, 
	$o'_n$ is contained in the unique block of pieces,
	$$o'_n\in W_n(\beta_0)\subset M'_n,$$
	where $\beta_0\in\mathscr{B}_n$ is the unique block of level $(n-1)$;
	for the blocks $\beta\in\mathscr{B}_n$ of level $\infty$,
	the corresponding $W_n(\beta)$ are all isomorphic to $V$.
	
	\subsubsection{Boundary-characteristic finite covers for blocks of pieces}
	In literature, \cite{DLW,PW-embedded} for example,
	boundary-characteristic finite covers 
	have been considered for JSJ pieces of irreducible
	closed $3$-manifolds to construct interesting finite covers.
	We consider a very similar situation
	where $V$ and $S$ play the roles
	of a JSJ piece and a JSJ torus accordingly.
	Continue to consider the tower of finite cyclic covers of $M$ encoded by
	the generalized digital expansion with respect to $d$.
	
	For any block of pieces $W_n(\beta)$, we say that a (possibly disconnected) 
	finite cover $\tilde{\mathcal{W}}$ of $W_n(\beta)$ is \emph{$\tilde{S}$--boundary-characteristic},
	if $\tilde{S}$ is a characteristic finite cover of $S$
	and 
	if every boundary component of $\tilde{\mathcal{W}}$ is isomorphic to $\tilde{S}$,
	as a cover of $S$ given by the composition $\tilde{\mathcal{W}}\to W_n(\beta)\to M$.
	Recall that 
	a characteristic cover $\tilde{X}$ of a path-connected space $X$ is a covering space of $X$
	which corresponds to a characteristic subgroup of $\pi_1(X)$, namely, 
	a subgroup invariant under action of the automorphism group $\mathrm{Aut}(\pi_1(X))$.
	If $\tilde{X}$ is a characteristic cover of some characteristic cover of $X$, it
	is also characteristic over $X$.
	

	The following omnipotence lemma is a consequence of 
	the Malnormal Special Quotient Theorem due to D.~T.~Wise
	\cite[Theorem 12.3]{Wise-book}, (see \cite{AGM-alternateMSQT} for an alternate proof).
	This lemma allows us to produce boundary-characteristic finite covers as deep as we wish.
	The assumption can certainly be replaced
	by a weaker one that $M$ is hyperbolic and $V$ has incompressible acylindrical boundary.

	\begin{lemma}\label{omnipotence}
		If $M$ is hyperbolic and $S$ is a retract of $M$,
		then for every finite cover $\mathcal{W}''$ of a block of pieces $W_n(\beta)$,
		there exists a characteristic finite cover $\tilde{S}^*$ of $S$ such that the following statement holds true:
		
		For every characteristic finite cover $\tilde{S}$ of $\tilde{S}^*$, 
		there exists 
		a regular finite cover $\tilde{\mathcal{W}}$ of $W_n(\beta)$
		which is $\tilde{S}$--boundary-characteristic.
		Moreover, the covering projection of $\tilde{\mathcal{W}}$ to $W_n(\beta)$
		factors through $\mathcal{W}''$.
	\end{lemma}
	
	\begin{proof}
		Observe that it suffices to prove for any connected $\mathcal{W}''$,
		otherwise taking $\tilde{S}^*$ to be a common characteristic finite cover of those 
		constructed componentwise. If $\beta\in\mathscr{B}_n$ is a block of level $j$ other than $\infty$, 
		the block of pieces $W_n(\beta)$	is homeomorphic to $M'_j$ removing a lift of $S$. 
		Therefore, it suffices to prove for $W_n(\beta)\cong V$, 
		otherwise arguing using $M'_j$ instead of $M$.
		We may also assume that any given $\mathcal{W}''$ 
		is characteristic over $V$, 
		otherwise replacing it with a further one such.
				
		Because $M$ is atoroidal and $S$ is a retract of $M$,
		the inclusion of $S$ into $M$ induces an embedding of $H=\pi_1(S)$ 
		into $\pi_1(M)$ 
		is as (a representative of the conjugacy class of) 
		a malnormal subgroup.
		Accordingly, the peripheral subgroups
		$H_\pm=\pi_1(\partial_\pm V)$ of the word hyperbolic group $G=\pi_1(V)$
		form a malnormal pair of quasi-convex subgroups.
		By \cite[Theorem 16.6]{Wise-book}, $G$ is virtually special.
		Moreover, it is implied by Wise's Malnormal Special Quotient Theorem 
		\cite[Theorem 12.3]{Wise-book} that there exist finite-index subgroups
		$\tilde{H}^*_\pm$ of $H_\pm$ with the following property:
		For any further finite-index subgroup $\tilde{H}_\pm$ of $\tilde{H}^*_\pm$,
		the quotient $G\,/\,\langle\langle \tilde{H}_+,\tilde{H}_-\rangle\rangle$,
		of $G$ by the normal closure of $\tilde{H}_\pm$,
		is word hyperbolic and virtually special. 
		In particular, it is residually finite.
		Without loss of generality,
		we may assume that $\tilde{H}^*_\pm$ are chosen to be characteristic in $H_\pm$,
		and isomorphic to the same characteristic finite-index subgroup $\tilde{H}^*$ of $H$.
		Since the given $G''=\pi_1(\mathcal{W}'')$ is characteristic
		as we have assumed in addition, 
		we may also assume that $\tilde{H}^*$ is chosen so deep
		that any conjugate of $H_\pm$ intersects $G''$ in exactly 
		the corresponding conjugate of $\tilde{H}^*_\pm$.
		Finally, we take the asserted characteristic finite cover $\tilde{S}^*$
		of $S$ to be the one corresponding to $\tilde{H}^*$.
		
		To verify the stated property,
		for any characteristic cover $\tilde{S}$ of $\tilde{S}^*$,
		denote by $\tilde{H}_\pm\cong\tilde{H}=\pi_1(\tilde{S})$ 
		the corresponding subgroup of	$\tilde{H}^*_\pm\cong \tilde{H}^*=\pi_1(\tilde{S}^*)$.
		Using the residual finiteness of $G\,/\,\langle\langle \tilde{H}_+,\tilde{H}_-\rangle\rangle$,
		we can find a finite-index normal subgroup $\tilde{G}$ of $G$
		which intersects $H_\pm$ in exactly $\tilde{H}_\pm$.
		Take $\tilde{G}''$ to be $\tilde{G}\cap G''$.
		Then any conjugate of $\tilde{H}_\pm$ intersects
		$\tilde{G}''$ in exactly the corresponding conjugate of $\tilde{H}_\pm$ as well.
		Therefore, the regular finite cover $\tilde{\mathcal{W}}$ of $V$
		corresponding to $\tilde{G}$ has every boundary component
		isomorphic to $\tilde{S}$, and the covering projection
		$\tilde{\mathcal{W}}\to V$ factors through the given intermediate cover $\mathcal{W}''$.		
	\end{proof}

	\subsection{Construction of the asserted tower}
	We construct the asserted tower of Proposition \ref{exponentialTorsionGrowthSpecial}
	adopting the notations and assumptions there.
	
	Given a constant $0<\epsilon<1$, choose a sufficiently large
	odd positive integer $d$ as provided by Lemma \ref{heavyDust}.
	The notations such as $s_n$, $[d^n]$, $\mathscr{B}_n$ from the generalized digital expansion 
	with respect to $d$	remains effective for the rest of this section.
	Note that the existence of the retract $\rho$ forces the oriented connected subsurface $S$
	to be non-separating. Denote by
		$$\cdots\longrightarrow (M'_n,o'_n)\longrightarrow \cdots\longrightarrow (M'_2,o'_2)\longrightarrow (M'_1,o'_1)\longrightarrow (M,*)$$
	the tower of finite cyclic covers dual to $S$, of degree $d^{s_n}$ over $M$, and with lifted base points, as before.
	
	Take a sequence which includes all the elements of $\pi_1(M,*)$, denoted by
		$$\{\,g_n\in\pi_1(M,*)\,\}_{n\in\Natural}.$$
	By inserting trivial elements between terms,
	we may assume that the sequence 
	satisfies the following additional properties:
	\begin{itemize}
	\item The first element $g_1$ is trivial.
	\item For every $n\in\Natural$, the algebraic intersection number 
	$\langle [g_n], [S]\rangle$ is bounded strictly by $d^{s_n}$ in absolute value.
	Here $[g_n]\in H_1(M;\Integral)$ and $[S]\in H_2(M;\Integral)$ 
	are the homology classes accordingly.
	\item Furthermore, if $\langle [g_n], [S]\rangle$ equals $0$, then the based lift of $g_n$ 
	in $M'_n$ is contained in the base block of pieces $W_n(\beta_0)$, up to based homotopy. 
	\end{itemize}
	
	Under the above setting, the asserted tower can be constructed by the following lemma.
	In the context of coverings spaces,
	the term \emph{elevation} is customarily used 
	to mean a preimage component of a sub-manifold
	in the referred cover, 
	so as to distinguish from the more common term \emph{lift},
	which is equivalently a homeomorphic elevation.
	
	\begin{lemma}\label{constructionAssertedTower}
	Under the assumptions of Proposition \ref{exponentialTorsionGrowthSpecial} and
	with the notations above,
	there are finite connected characteristic covers $\{\tilde{S}_n\to S\}_{n\in\Natural}$
	and base-pointed connected finite (irregular) covers 
	$\{\,(\tilde{M}_n,\tilde{o}_n)\to(M'_n,o'_n)\,\}_{n\in\Natural}$
	such that the following properties are satisfied:
	\begin{itemize}
	\item Each $\tilde{M}_{n+1}$ is a finite cover of $\tilde{M}_n$,
	and the covering maps fit into the commutative diagram:
	$$\begin{CD}
	\tilde{M}_{n+1} @>>> \tilde{M}_n\\
	@VVV @VVV\\
	M'_{n+1} @>>> M'_n
	\end{CD}$$
	\item 
	The element $g_n$ does not lift into $\tilde{M}_n$ based at $\tilde{o}_n$ for any $n\in\Natural$,
	unless it is trivial.
	\item 
	For each block $\beta\in\mathscr{B}_n$ of level $\infty$
	and for every elevation 
	$\tilde{W}$  in $\tilde{M}_n$ of $W_n(\beta)\cong V$, 
	the induced covering map $\tilde{W}\to V$
	is isomorphic to 
	the pull-back of the covering map $\tilde{S}_n\to S$ via the composition:
	$$\begin{CD}
	V @>\mathrm{incl.}>> M @>\rho>> S.
	\end{CD}$$	
	\end{itemize}
	
	\end{lemma}
	
	\begin{proof}
	We construct $\tilde{S}_n$ and $(\tilde{M}_n,\tilde{o}_n)$ by induction on $n\in\Natural$.
	For $n$ equal to $1$, we can simply take $\tilde{S}_1$ to be $S$ and $(\tilde{M}_1,\tilde{o}_1)$ to be $(M'_1,o'_1)$.
	Suppose that all its previous stages have been completed,
	then we proceed with the $n$-th stage, where $n$ is greater than $1$.
	
	We introduce a few notations. Denote
		$$(\mathcal{M}^\times_n,\,o^\times_n)\,=\,\left(M'_n\times_{M'_{n-1}} \tilde{M}_{n-1},\,(o'_n,\tilde{o}_{n-1})\right),$$
	the fiber product of base-pointed covering spaces.
	The space $\mathcal{M}^\times_n$ is a possibly disconnected base-pointed finite cover of $M'_{n-1}$
	that factors through both $M'_n$ and $\tilde{M}_{n-1}$.
	It can be concretely described as follows:
	As $d$ is an odd number, there is a unique lift $S'_{n-1}$ of $S$ in $M'_{n-1}$ 
	furthermost from $o'_{n-1}$; 
	therefore, $\mathcal{M}^\times_n$ can be obtained as a cyclic cover of $\tilde{M}_{n-1}$
	of degree $d^n$, which is dual to the preimage of $S'_{n-1}$.
	For any block $\beta\in\mathscr{B}_n$, denote by $\mathcal{W}^\times_n(\beta)$
	the preimage of the block of pieces $W_n(\beta)\subset M'_n$.
	Observe that $\mathcal{W}^\times_n(\beta)$ are all $\tilde{S}_{n-1}$--boundary-characteristic
	over $W_n(\beta)$. 
	In fact,  by the induction hypothesis,
	for any block $\gamma\in\mathscr{B}_{n-1}$ of level $\infty$,
	every elevation $\tilde{W}$ of $W_{n-1}(\gamma)$ in $\tilde{M}_{n-1}$
	is all $\tilde{S}_{n-1}$--boundary-characteristic.
	Then the observation follows as every component of $\mathcal{W}^\times_n(\beta)$
	is next to a lift of such $\tilde{W}$.
	Denote the disjoint union of all $\mathcal{W}^\times_n(\beta)$ by
		$$\mathcal{W}^\times_n\,=\,\bigsqcup_{\beta\in\mathscr{B}_n}\,\mathcal{W}^\times_n(\beta).$$
	The gluing is formally given by an orientation-reversing free involution:
		$$\begin{CD}\partial\mathcal{W}^\times_n @> \nu^\times >>\partial\mathcal{W}^\times_n\end{CD}$$
	which commutes with the underlying gluing identification between the components of all $\partial W_n(\beta)$
	via the covering projection.
	Identifying each orbit of $\nu^\times_n$ to a point yields
	$$\mathcal{M}^\times_n\,=\,\mathcal{W}^\times_n\,/\,\nu^\times_n.$$
	
	There are two simple cases where the construction of $\tilde{S}_n$ and $(\tilde{M}_n,\tilde{o}_n)$
	is straightforward. If the element $g_n\in\pi_1(M,*)$ is trivial, $\tilde{S}_n$ can be taken simply
	as $\tilde{S}_{n-1}$, and $(\tilde{M}_n,\tilde{o}_n)$ can be taken
	as the base component of $(\mathcal{M}^\times_n,o^\times_n)$. Similarly, if the algebraic intersection number
	$\langle [g_n],[S]\rangle$ in $M$ is nontrivial, hence less than $d^{s_n}$
	in absolute value.
	By the assumptions on $\{g_n\}_{n\in\Natural}$, 
	the element $g_n$ does not lift to $M'_n$ or $\mathcal{M}^\times_n$.
	So again, we can take 
	$\tilde{S}_n$ and $(\tilde{M}_n,\tilde{o}_n)$ the same as the trivial case.
	
	It remains to prove the essential case when $g_n$ is nontrivial with $\langle [g_n],[S]\rangle=0$ in $M$. 
	In this case, $g_n$ lifts to be a nontrivial element of $\pi_1(M'_n,o'_n)$.
	Moreover, abusing the notation, we have 
	$$g_n\in \pi_1(W_n(\beta_0),o'_n)$$ 
	by the assumption of $\{g_n\}_{n\in\Natural}$.
	Denote by $\mathcal{W}^\times_n(\beta_0)$ the preimage of $W_n(\beta_0)$ in $\mathcal{M}^\times_n$.
	By the residual finiteness of $\pi_1(W_n(\beta_0),o'_n)$, there exists a further regular finite cover 
	$\mathcal{W}''_n(\beta_0)$ of $W_n(\beta_0)$ which factors through
	$\mathcal{W}^\times_n(\beta_0)$, such that $g_n$ does not lift to $\mathcal{W}''_n(\beta_0)$.
	By choosing any base point $o''_n$ lifting $o^\times_n$, we have
		$$g_n\not\in\pi_1(\mathcal{W}''_n(\beta_0),o''_n).$$
		
	Note that the furthermost lift $S'_{n-1}$ of $S$ in $M'_{n-1}$ above has the property that
	its complement lifts into	$M'_n$ as the interior of $W_n(\beta_0)$.
	Moreover, $S'_{n-1}$ is a retract of $M'_{n-1}$,
	since $S$ is a retract of $M$ as assumed by Proposition \ref{exponentialTorsionGrowthSpecial}.
	By applying Lemma \ref{omnipotence} with respect to $M'_{n-1}$, $S'_{n-1}$ and $\mathcal{W}''_n(\beta_0)$,
	we obtain a finite characteristic cover $\tilde{S}^*_n(\beta_0)$ of $S'_{n-1}\cong S$.
	It has the property that for any further finite characteristic cover of $S$,
	boundary-characteristic finite covers of $\mathcal{W}''_{n-1}(\beta_0)$ with the prescribed boundary pattern
	exist, and can be constructed to factor through $\mathcal{W}''_n(\beta_0)$.	
	Similarly, for each block $\beta\in\mathscr{B}_n$ of level $j$ other than $\infty$ or $(n-1)$,
	apply Lemma \ref{omnipotence} 
	with respect to $M'_j$, $S'_j$ and $\mathcal{W}^\times_n(\beta)$ to obtain 
	a finite characteristic cover $\tilde{S}^*_n(\beta)$ of $S$;
	for each $\beta\in\mathscr{B}$ of level $\infty$,
	we can simply take $\tilde{S}^*_n(\beta)$ to be $\tilde{S}_{n-1}$.
	We take a finite characteristic cover
		$$\tilde{S}^*_n\longrightarrow S$$
	which factors through $\tilde{S}^*_n(\beta)$ for all $\beta\in\mathscr{B}_n$.
	Note that $\tilde{S}^*_n\to S$ factors through $\tilde{S}_{n-1}$.
	
	To construct the claimed $\tilde{S}_n$ and $(\tilde{M}_n,\tilde{o}_n)$,
	we first take the claimed finite characteristic cover
		$$\tilde{S}_n\longrightarrow S$$
	to be a finite characteristic cover of $\tilde{S}_{n-1}$ which factors through $\tilde{S}^*_n$.
	The claimed base-pointed finite cover $(\tilde{M}_n,\tilde{o}_n)$ can be constructed
	by merging $\tilde{S}_n$--boundary-characteristic covers of $\mathcal{W}^\times_n(\beta)$ as follows.
	
	For each block $\beta\in\mathscr{B}_n$,
	we first construct an $\tilde{S}_n$--boundary-characteristic
	cover $\tilde{\mathcal{W}}_n(\beta)$ of $\mathcal{W}_n^\times(\beta)$,
	in the following way.
	For the block $\beta_0\in\mathscr{B}_n$ of level $(n-1)$,
	take an $\tilde{S}_n$--boundary-characteristic finite cover of 
	$\mathcal{W}^\times_n(\beta_0)$	which factors through $\mathcal{W}''_n(\beta_0)$,
	denoted as $\tilde{\mathcal{W}}_n(\beta_0)$. 
	Similarly, for each block $\beta\in\mathscr{B}_n$ of level $j$
	other than $\infty$ or $(n-1)$, 
	take an $\tilde{S}_n$--boundary-characteristic finite cover of 
	$\mathcal{W}^\times_n(\beta)$. 
	For each block $\beta\in\mathscr{B}_n$ of level $\infty$,
	we need a more specific construction of $\tilde{\mathcal{W}}_n(\beta)$ 
	to meet the third property of Lemma \ref{constructionAssertedTower}.
	Note that	$\mathcal{W}^\times_n(\beta)$
	is isomorphic to 
	$\tilde{\mathcal{W}}_{n-1}(\gamma)\times[d^n]$ for some level--$\infty$ block
	$\gamma\in\mathscr{B}_{n-1}$.
	Therefore, by the induction hypothesis,
	any component $W^\times$ of $\mathcal{W}^\times_n(\beta)$,
	as a covering space of $W_n(\beta)\cong V$,
	is isomorphic to the pull-back via $\tilde{S}_{n-1}\to S$ of 
	$$\begin{CD}V @>\rho_V >> S,\end{CD}$$
	the composition of the retraction $\rho$ of $M$ to $S$ with
	the inclusion of $V$ into $M$. In other words,
	it is isomorphic to a fiber product:
		$$W^\times\cong V\times_{\rho_V} \tilde{S}_{n-1}.$$
	Denote the $\tilde{S}_n$--boundary-characteristic cover of $W_n(\beta)\cong V$:
		$$\tilde{W}\cong V\times_{\rho_V} \tilde{S}_n,$$
	then $\tilde{W}$ is isomorphic to the pull-back of $\tilde{S}_n$
	via $\rho_V$ and the covering projection factors through $W^\times$.
	As $W^\times$ runs over the components of $\mathcal{W}^\times(\beta)$,
	take $\tilde{\mathcal{W}}_n(\beta)$
	for the level--$\infty$ block $\beta$ to be the disjoint union of all the $\tilde{W}$
	accordingly.
	Let us formally
	put together all the building parts $\tilde{\mathcal{W}}_n(\beta)$ that we have constructed above:
		$$\tilde{\mathcal{W}}_n\,=\,\bigsqcup_{\beta\in\mathscr{B}_n}\,\tilde{\mathcal{W}}_n(\beta),$$
	
	At this point, we may not yet be able to obtain a closed manifold
	by assembling the components of $\tilde{\mathcal{W}}_n$,
	because the boundary components of $\tilde{\mathcal{W}}_n$
	that we wanted to glue up together 
	may not be balanced in amount between opposite orientations.
	We need to suitably duplicate the components of $\tilde{\mathcal{W}}_n$
	to meet the balance condition.
	A simple solution is to introduce a quantity
		$$K_n\,=\,K_n(\tilde{\mathcal{W}}_n\to \mathcal{W}^\times)$$
	which is defined to be the least common multiple of all the local covering degrees
	$[\tilde{\mathcal{W}}_n:\mathcal{W}^\times_n]_{W^\times}$ where $W^\times$
	runs over all the components of $\mathcal{W}^\times_n$.
	Here the (unsigned) local covering degree $[\tilde{\mathcal{W}}_n:\mathcal{W}^\times_n]_{W^\times}$
	is defined to be 
	the number of lifts in $\tilde{\mathcal{W}}_n$ for any point of $W^\times$.
	Replace the preimage of each component $W^\times$ by 
	the disjoint union of $K_n\,/\,[\tilde{\mathcal{W}}_n:\mathcal{W}^\times_n]_{W^\times}$ copies of itself.
	It is easy to see that the new cover $\tilde{\mathcal{W}}_n$ has constant local degree
	$[\tilde{\mathcal{W}}_n:\mathcal{W}^\times_n]_{W^\times}\,=\,K_n$ for all $W^\times$.
	So the balance condition is satisfied as $\tilde{\mathcal{W}}_n$ is already boundary-characteristic
	over $\mathcal{W}^\times_n$.
	
	We can construct the claimed $\tilde{M}_n$ 
	by gluing up the components of $\tilde{\mathcal{W}}_n$ along boundary.
	There is a fairly routine procedure to do so.
	We provide some details below
	for the reader's reference.
	(See \cite{DLW,PW-embedded} for similar constructions with respect to JSJ decompositions.)
	The balance condition allows us to construct
	a free involution $\tilde{\nu}_\sharp$ which pairs up oppositely oriented boundary components 
	and commutes with the pairing free involution $\nu^\times_\sharp$ via the covering projection,
	namely, the following diagram commutes:
	$$\begin{CD}
	\pi_0(\partial\tilde{\mathcal{W}}_n) @> \tilde{\nu}_\sharp >> \pi_0(\partial\tilde{\mathcal{W}}_n)\\
	@VVV @VVV\\
	\pi_0(\partial\mathcal{W}^\times_n) @> \nu^\times_\sharp >> \pi_0(\partial\mathcal{W}^\times_n)\\
	\end{CD}$$
	Furthermore, suppose that $\tilde{P}_\pm\in \pi_0(\partial\tilde{\mathcal{W}}_n)$ is a pair of oppositely
	oriented boundary components whose projection $P^\times_\pm\in \pi_0(\partial\mathcal{W}^\times_n)$
	satisfies $\nu^\times_\sharp(P^\times_+)=P^\times_-$.
	Since the covering projections $\tilde{P}_\pm\to P^\times_\pm$ are characteristic modeled on
	$\tilde{S}_n\to\tilde{S}_{n-1}$, we can promote the gluing free involution $\nu^\times|_{P^\times_\pm}$ to
	$\tilde\nu_{\tilde{P}_\pm}$, namely, the following diagram commutes:
	$$\begin{CD}
	\tilde{P}_\pm @> \tilde{\nu} >> \tilde{P}_\mp\\
	@VVV @VVV\\
	P^\times_\pm @> \nu^\times >> P^\times_\mp\\
	\end{CD}$$
	Promoting $\nu^\times$ to all the components of $\partial\tilde{\mathcal{W}}_n$,
	which have been paired up by $\tilde\nu_\sharp$, we obtain an orientation-reversing free involution
	$$\begin{CD}\partial\tilde{\mathcal{W}}_n @> \tilde\nu >> \partial\tilde{\mathcal{W}}_n.\end{CD}$$
	
	Finally, we take the claimed finite cover of $M'_n$ to be 
	$$\tilde{M}_n\,=\,\tilde{\mathcal{W}}_n\,/\,\tilde{\nu}.$$
	The claimed base point $\tilde{o}_n$ of $\tilde{M}_n$ can be chosen as any lift of $o'_n$ in $\tilde{\mathcal{W}}_n(\beta_0)$.
	It is clear from the construction that
	the claimed properties of Lemma \ref{constructionAssertedTower} are satisfied by
	$\tilde{S}_n$ and $(\tilde{M}_n,\tilde{o}_n)$. This completes the induction.
	\end{proof}

	\subsection{Verification of the asserted properties}
	To briefly summarize what we have done so far, under the assumptions of Proposition \ref{exponentialTorsionGrowthSpecial},
	for any given constant $0<\epsilon<1$, a sufficiently large positive integer $d$ has been chosen
	as guaranteed by Lemma \ref{heavyDust}. Moreover, the following commutative diagram
	of covering maps between base-pointed covers of $M$ has been constructed by Lemma \ref{constructionAssertedTower},
	where the upper row is the asserted tower of Proposition \ref{exponentialTorsionGrowthSpecial}
	and the lower row is the cyclic tower dual to $S$ encoded by the generalized digital expansion
	with respect to $d$:
	$$\begin{CD}
	\cdots @>>> (\tilde{M}_n,\tilde{o}_n) @>>> \cdots @>>> (\tilde{M}_2,\tilde{o}_2) @>>> (\tilde{M}_1,\tilde{o}_1) @>>> (M,*)\\
	@. @VVV @. @VVV @VVV @VV \mathrm{id} V\\
	\cdots @>>> (M'_n,o'_n) @>>> \cdots @>>> (M'_2,o'_2) @>>> (M'_1,o'_1) @>>> (M,*)
	\end{CD}$$
	As before, we denote by $W_n(\beta)$ the block of pieces of $M'_n$ and $\tilde{\mathcal{W}}_n(\beta)$
	their preimage in $\tilde{M}_n$ accordingly.
	
	It remains to verify the requirements of Proposition \ref{exponentialTorsionGrowthSpecial}
	are satisfied.
	The injectivity radii of $\tilde{M}_n$ at $\tilde{o}_n$ tends to infinity as $n$ grows,
	because of the second property of Lemma \ref{constructionAssertedTower}.
	The estimate of lifts of $\iota_\Sigma$ is a consequence of the third property
	of Lemma \ref{constructionAssertedTower}. In fact, for every block $\beta\in\mathscr{B}_n$
	of level $\infty$, any component $\tilde{W}$ of $\tilde{\mathcal{W}}_n$ 
	is a regular finite cover of $V$
	isomorphic the pull-back of the characteristic finite cover $\tilde{S}_n\to S$ via 
	the composition $\rho_V\colon V\to S$
	of the retraction $\rho$ of $M$ to $S$ with the inclusion of $V$ into $M$.
	Since $\rho\circ\iota_\Sigma$ is homotopic to a constant point map,
	every elevation of $\Sigma$ into $\tilde{W}$ is a lift, and the number of such lift
	is the covering degree $[\tilde{S}_n:S]$.
	Therefore, the number of lifts of $\iota_\Sigma$
	into $\tilde{\mathcal{W}}_n(\beta)$, for any level--$\infty$
	blocks $\beta\in\mathscr{B}_n$, equals 
	$$|\pi_0(\tilde{\mathcal{W}}_n(\beta))|\cdot[\tilde{S}_n:S]\,=\,[\tilde{M}_n:M'_n].$$
	The number of lifts of $\iota_\Sigma$ into $\tilde{M}$ is 
	at least the number of such lifts, which is at least
	$$(1-\epsilon)\cdot d^{s_n}\cdot [\tilde{M}_n:M'_n]\,=\,(1-\epsilon)\cdot[\tilde{M}_n:M],$$
	as asserted. 
	The estimate of the homological torsion size follows immediately from the topological fact:
	
	\begin{lemma}\label{oddChiAndTor}
		If $\Sigma_1,\cdots,\Sigma_m$ are mutually disjointly embedded closed subsurfaces of odd Euler characteristic
		in an orientable compact $3$-manifold $N$, then $H_1(N;\Integral)$ 
		contains a submodule isomorphic to $\Integral_2^{\oplus m}$.
	\end{lemma}
	
	\begin{remark}
		The author thanks I. Agol for pointing out to him
		the connection between this fact and exponential torsion growth in dual-graph covers. 
		Theorem \ref{exponentialTorsionGrowth}
		grows out of that interesting observation.
	\end{remark}
	
	\begin{proof}
		Each $[\Sigma_i]$ represents an element $[\Sigma_i]\in H^1(N;\Integral_2)$ via Poincar\'e duality.
		Since $\Sigma_i$ has odd Euler characteristic, in the cohomology ring of $\Integral_2$ coefficients,
		$[\Sigma_i]^3$ is nontrivial by \cite[Theorem 4.1]{HWZ-lens},
		but $[\Sigma_i][\Sigma_j]$ is $0$ for every distinct pair $i,j$ by the mutual disjointness.
		It follows that $[\Sigma_i]$ cannot be be lifted to $H^1(N;\Integral)$
		but	they span an $m$-dimensional subspace of $H^1(N;\Integral_2)$.
		By the Universal Coefficient Theorem,
		$H_1(N;\Integral)$ contains a torsion submodule isomorphic to $\Integral_2^{\oplus m}$.
	\end{proof}
	
	This completes the proof of Proposition \ref{exponentialTorsionGrowthSpecial}.
	
\section{Application to uniform lattices of $\PSL(2,\Complex)$}\label{Sec-uniformLattices}

	In this section, we derive the existence of exhausting nested sequence of
	finite index subgroup with exponential homological torsion growth
	for uniform lattices of $\PSL(2,\Complex)$,
	proving Theorem \ref{exponentialTorsionGrowth}.
	
	Let $\Gamma$ be any uniform lattice of $\PSL(2,\Complex)$.
	Take a torsion-free finite-index subgroup $\dot{\Gamma}$ of $\Gamma$,
	so the quotient space 
		$$N=\Hyp\,/\,\dot{\Gamma}$$
	is an orientable closed hyperbolic $3$-manifold.
	
	By \cite{Agol-VHC} and \cite[Theorem 16.6]{Wise-book},
	we may assume that $\pi_1(N)\cong\dot{\Gamma}$ is cocompactly special,
	namely, it admits a special cocompact action on a CAT(0) cube complex.
	By Theorem \ref{main-oddEulerCharacteristic}, there is a connected closed
	surface of odd Euler characteristic $\Sigma$ 
	which admits a $\pi_1$--injective quasi-Fuchsian immersion into $N$.
	Take $S$ to be an orientable connected closed surface 
	which also admits a $\pi_1$--injective quasi-Fuchsian immersion into $N$,
	for example, some finite cover of $\Sigma$.
	Form a $2$-complex $X$ by attaching an arc $\alpha$ between 
	$S$ and $\Sigma$, so
		$$\pi_1(X)\cong\pi_1(S)*\pi_1(\Sigma).$$
	By well known constructions, we can find a map
		$$f\colon X\longrightarrow N$$
	which embeds $\pi_1(X)$ into $\pi_1(N)$ as a quasi-convex subgroup.
	Moreover, we can require the restrictions
	to $S$ and $\Sigma$ to be the claimed quasi-Fuchsian immersions above.
	
	Since $\pi_1(N)$ is word hyperbolic and special, and $\pi_1(X)$ is embedded
	as a quasi-convex subgroup, it follows from \cite[Theorem 7.3]{Haglund-Wise}
	(see also \cite[Theorem 4.13]{Wise-book}) that 
	$\pi_1(X)$ is a virtual retract	of $\pi_1(N)$.
	In terms of maps, 
	there exists a finite cover 
		$$M\to N$$ into which $f$ lifts
	to be a map
	$$\iota_X\colon X\longrightarrow M,$$
	and there exists a map
	$$\rho_X\colon M\longrightarrow X,$$
	such that $\rho_X\circ\iota_X$ is homotopic to the identity.
	In fact, it is not hard to argue that $\iota_X$ can be homotoped to be an embedding.
	Fix a generic base point $*$ of $M$.
	Denote by $\iota_S$ and $\iota_\Sigma$ the induced embeddings of $S$ and $\Sigma$
	into $M$, and by $\rho$ the composition of retraction maps:
	$$\begin{CD} M @> \rho_X >> X @> \mathrm{retr.} >> S.\end{CD}$$

	Finally, we apply Proposition \ref{exponentialTorsionGrowth} to conclude that $(M,*)$ admits
	an exhausting nested tower of base-pointed finite covers 
	with exponential homological torsion growth. Since $\pi_1(M,*)$ 
	is isomorphic to a finite index subgroup of $\Gamma$, 
	the same conclusion holds for $\Gamma$ as well. 
	This completes the proof of Theorem \ref{exponentialTorsionGrowth}.

\bibliographystyle{amsalpha}


\end{document}